\documentclass[11pt,reqno,oneside]{amsart}

\usepackage{graphicx}
\usepackage{amssymb}
\usepackage{epstopdf}

\usepackage[a4paper, total={6in, 9in}]{geometry}

\usepackage{amsmath,amsfonts,amsthm,mathrsfs,amssymb,cite}
\usepackage[usenames]{color}
\usepackage{cases}

\newtheorem{thm}{Theorem}[section]

\newtheorem{lem}{Lemma}[section]

\theoremstyle{definition}

\theoremstyle{remark}

\newtheorem{rem}{Remark}[section]
\numberwithin{equation}{section}

\usepackage{tikz}
\allowdisplaybreaks

\newcommand{\R}{\mathbb{R}}

\newcommand{\scrD}{\mathscr{D}}

\newcommand{\supp}{\mathop{\rm supp}}

\newcommand{\Forall}{{\,\forall\,}}
\newcommand{\Exists}{{\,\exists\,}}
\newcommand{\st}{\textrm{~s.t.~}}
\newcommand{\aee}{\textrm{~a.e.~}}
\newcommand{\ass}{\textrm{~a.s.~}}

\newcommand{\dif}[1]{\,\mathrm{d}{#1}} 
\newcommand{\nrm}[2][]{ \| {#2} \|_{#1}} 
\newcommand{\agl}[1][\cdot]{ \langle {#1} \rangle}

\newcommand{\Rk}{{\mathcal{R}_{k}}}

\newcommand{\jz}[1]{{#1}}
\newcommand{\sq}[1]{{#1}}
\newcommand{\hy}[1]{{#1}}

\newcounter{saveeqn}

\title[Determining a random Schr\"odinger equation]{Determining a random Schr\"odinger equation with unknown source and potential}

\author{Jingzhi Li}
\address{Department of Mathematics, Southern University of Science and Technology, Shenzhen, China}
\email{\jz{li.jz@sustech.edu.cn}}

\author{Hongyu Liu}
\address{Department of Mathematics, Hong Kong Baptist University, Kowloon, Hong Kong SAR, China}
\email{hongyu.liuip@gmail.com, hongyuliu@hkbu.edu.hk}

\author{Shiqi Ma}
\address{Department of Mathematics, Hong Kong Baptist University, Kowloon, Hong Kong SAR, China 
	and Department of Mathematics, Southern University of Science and Technology, Shenzhen, China}
\email{\jz{mashiqi01@gmail.com}}

\begin{document}

	\begin{abstract}
		
		\jz{We are concerned with the direct and inverse scattering problems associated with a time-harmonic random Schr\"odinger equation with unknown source and potential terms. 
		The well-posedness of the direct scattering problem is first established. 
		Three uniqueness results are then obtained for the corresponding inverse problems in determining the variance of the source, the potential and the expectation of the source, respectively, by the associated far-field measurements. 
		First, a single realization of the passive scattering measurement can uniquely recover the variance of the source without the a priori knowledge of the other unknowns. 
		Second, if active scattering measurement can be further obtained, a single realization can uniquely recover the potential function without knowing the source. 
		Finally, both the potential and the first two statistic moments of the random source can be uniquely recovered with full measurement data. The major novelty of our study is that on the one hand, both the random source and the potential are unknown, and on the other hand, both passive and active scattering measurements are used for the recovery in different scenarios.}

		\medskip

		\noindent{\bf Keywords:}~~random Schr\"odinger equation, inverse scattering, passive/active measurements, \jz{asymptotic expansion, ergodicity}
		
		{\noindent{\bf 2010 Mathematics Subject Classification:}~~35Q60, 35J05, 31B10, 35R30, 78A40}
		
	\end{abstract}

	\maketitle

	\section{Introduction} \label{sec:Intro-SchroEqu2018}

	In this paper, we are mainly concerned with the following random Schr\"odinger system
	\begin{subequations} \label{eq:1}
	\begin{numcases}{}
	\displaystyle{ (-\Delta-E+V(x)) u(x, E, d, \omega)= f(x)+\sigma(x)\dot{B}_x(\omega), \quad x\in\mathbb{R}^3, } \label{eq:1a} \medskip\\
	\displaystyle{ u(x, E, d, \omega)=\alpha e^{\mathrm{i}\sqrt E x\cdot d}+u^{sc}(x, E, d,\omega), } \label{eq:1b} \medskip\\
	\displaystyle{ \lim_{r\rightarrow\infty} r\left(\frac{\partial u^{sc}}{\partial r}-\mathrm{i}\sqrt{E} u^{sc} \right)=0,\quad r:=|x|, } \label{eq:1c}
	\end{numcases}
	\end{subequations}
	\hy{where $f(x)$ and $\sigma(x)$ in \eqref{eq:1a} are the expectation and standard variance of the source term, $d \in \mathbb{S}^2:=\{ x \in \R^3 \,;\, |x| = 1 \}$ signifies the impinging direction of the incident plane wave, and $E\in\mathbb{R}_+$ is  the energy level. 
	In \eqref{eq:1b}, $\alpha$ takes the value of either 0 or 1 to incur or suppress the presence of the incident wave, respectively. 
	In the sequel, we follow the convention to replace $E$ with $k^2$, namely $k := \sqrt{E} \in \R_+$, which can be understood as the wave number. 
	The limit in \eqref{eq:1c} is the Sommerfeld Radiation Condition (SRC) \cite{colton2012inverse} that characterizes the outgoing nature of the scattered wave field $u^{sc}$. The random system \eqref{eq:1} describes the quantum scattering associated with a potential $V$ and a random active source $(f, \sigma)$ at the energy level $k^2$.}
	
	\jz{In the system \eqref{eq:1}, the random parameter $\omega$ belongs to $\Omega$ with $(\Omega, \mathcal{F}, \mathbb{P})$ signifying a complete probability space.
	The term $\dot B_x(\omega)$ denotes the three-dimensional spatial Gaussian white noise \cite{dudley2002real}. 
	The random part $\sigma(x) \dot B_x(\omega)$ within the source term in \eqref{eq:1a} is an ideal mathematical model for noises arising from real world applications \cite{dudley2002real}.
	We note that the $\sigma^2(x)$ gives the intensity of the randomness of the source at the point $x$, and can be understood as the variance of $\sigma(x) \dot B(x,\omega)$. In what follows, we call $\sigma^2(x)$ the variance function. 
	The statistical information of a single zero-mean Gaussian white noise is encoded in its variance function \cite{ross2014introduction}.
	In this paper, we are mainly concerned with the recovery of the variance and expectation of the random source as well as the potential function in \eqref{eq:1} by the associated scattering measurements as described in what follows.}

	\hy{In order to study the corresponding inverse problems, one needs to have a thorough understanding of the direct scattering problem. In the deterministic case with $\sigma\equiv 0$, the scattering system \eqref{eq:1} is well understood; see, e.g., \cite{colton2012inverse,griffiths2016introduction}. There exists a unique solution $u^{sc}\in H^1_{loc}(\mathbb{R}^3)$, and moreover there holds the following asymptotic expansion as $|x|\rightarrow\infty$, 
	\begin{equation}\label{eq:farfield}
	u^{sc}(x) = \frac{e^{\mathrm{i}k r}}{r} u^\infty(\hat x, k, d) + \mathcal{O} \left( \frac{1}{r^2} \right),
	\end{equation}
	where $\hat x := x/{|x|} \in \mathbb{S}^2$. The term $u^\infty$ is referred to as the far-field pattern, which encodes the information of the potential $V$ and the source $f$. 
	In principle, we shall show that the random scattering system \eqref{eq:1} is also well-posed in a proper sense and possesses a far-field pattern. 
	To that end, throughout the rest of the paper, we assume that $\sigma^2$, $V$, $f$ belong to $L^\infty(\mathbb{R}^3;\R)$, respectively, and that they are compactly supported in a fixed bounded domain $D\subset \mathbb{R}^3$ containing the origin. 
	Under the aforementioned regularity assumption, we establish that the following mapping of the direct problem (\textbf{DP}) is well-posed in a proper sense,
	\begin{equation} \label{eq:dp-SchroEqu2018}
	\textbf{DP \ :} \quad (\sigma, V, f) \rightarrow \{u^{sc}(\hat x, k, d, \omega), u^\infty(\hat x, k, d, \omega) \,;\, \omega \in \Omega,\, \hat x \in \mathbb{S}^2, k > 0,\,  d \in \mathbb{S}^2 \}.
	\end{equation}

  	The well-posedness of the direct scattering problem paves the way for our further study of the inverse problem ({\bf IP}).}
  	\sq{In {\bf IP}, we are concerned with the recoveries of the three unknowns $\sigma^2$, $V$, $f$ in a \emph{sequential} way, by knowledge of the associated far-field pattern measurements $u^\infty(\hat x, k, d, \omega)$. 
  	By sequential, we mean the $\sigma^2$, $V$, $f$ are recovered by the corresponding data sets one-by-one.  	
  	In addition to this, in the recovery procedure, both the \emph{passive} and \emph{active} measurements are utilized. 
  	When $\alpha = 0$, the incident wave is suppressed and the scattering is solely generated by the unknown source. The corresponding far-field pattern is thus referred to as the passive measurement. 
  	In this case, the far-field pattern is independent of the incident direction $d$, and we denote it as $u^\infty(\hat x, k, \omega)$.  	
  	When $\alpha = 1$, the scattering is generated by both the active source and the incident wave, and the far-field pattern is referred to as the active measurement, denoted as $u^\infty(\hat x, k, d, \omega)$.
  	Under these settings, we formulate our {\bf IP} as}
  	\sq{\begin{equation}\label{eq:ip-SchroEqu2018}
  	\textbf{IP \ :}\quad \left\{
  	\begin{aligned}
  	\mathcal M_1(\omega) := & \ \{u^\infty(\hat x, k, \omega) \,;\, \forall \hat x \in \mathbb{S}^2,\, \forall k \in \R_+ \} && \rightarrow \quad \sigma^2, \\
  	\mathcal M_2(\omega) := & \ \{u^\infty(\hat x, k, d, \omega) \,;\, \forall \hat x \in \mathbb{S}^2,\, \forall k \in \R_+,\, \forall d \in \mathbb{S}^2 \} && \rightarrow \quad V, \\
  	\mathcal M_3 := & \ \{u^\infty(\hat x, k, d, \omega) \,;\, \forall \hat x \in \mathbb{S}^2,\, \forall k \in \R_+,\, d\  \text{fixed},\, \forall \omega \in \Omega\, \} && \rightarrow \quad f.
  	\end{aligned}
  	\right.
  	\end{equation}
	The data set $\mathcal M_1(\omega)$ (abbr.~$\mathcal M_1$) corresponds to the passive measurement ($\alpha = 0$), while the data sets $\mathcal M_2(\omega)$ (abbr.~$\mathcal M_2$) and $\mathcal M_3$ correspond to the active measurements ($\alpha = 1$).
	Different random sample $\omega$ gives different data sets $\mathcal M_1$ and $\mathcal M_2$.
	All of the $\sigma^2$, $V$, $f$ in the {\bf IP} are assumed to be unknown, and our study shows that the data sets $\mathcal M_1$, $\mathcal M_2$, $\mathcal M_3$ can recover $\sigma^2$, $V$, $f$, respectively. 
	The mathematical arguments of our study are constructive and we derive explicitly recovery formulas, which can be employed for numerical reconstruction in future work.}

	In the aforementioned {\bf IP}, we are particularly interested in the case with a single realization, namely the sample $\omega$ is fixed in the recovery of $\sigma^2$ and $V$ in \eqref{eq:ip-SchroEqu2018}.
 	\sq{Intuitively, a particular realization of $\dot B_x$ provides little information about the statistical properties of the random source. 
 	However, our study indicates that a \emph{single realization} of the far-field measurement can be used to uniquely recover the variance function and the potential in certain scenarios. A crucial assumption to make the single-realization recovery possible is that the randomness is independent of the wave number $k$. Indeed, 
	there are assorted applications in which the randomness changes slowly or is independent of time \cite{caro2016inverse, Lassas2008}, and by Fourier transforming into the frequency domain, they actually correspond to the aforementioned situation. 
	The single-realization recovery has been studied in the literature; 
	see, e.g., \cite{caro2016inverse,Lassas2008,LassasA}.
	The idea of this work is mainly motivated by \cite{caro2016inverse}.}

	There are abundant literatures  for the inverse scattering problem associated with either the passive or active measurements.  Given an known potential, the recovery of an unknown source term by the corresponding passive measurement is referred to as the  inverse source problem. We refer to \cite{bao2010multi,Bsource,BL2018,ClaKli,GS1,Isakov1990,IsaLu,Klibanov2013,KS1,WangGuo17,Zhang2015} and the references therein for both theoretical uniqueness/stability results and computational methods for the inverse source problem in the deterministic setting, namely $\sigma\equiv 0$. 
	\sq{The authors are also aware of some study on the inverse source problem concerning the recovery of a random source  \cite{LiLiinverse2018,LiHelinLiinverse2018}.
	In \cite{LiHelinLiinverse2018}, the homogeneous Helmholtz system with a random source is studied.
	Compared with \cite{LiHelinLiinverse2018}, our system \eqref{eq:1} comprises of both unknown source and unknown potential, which make the corresponding study radically more challenging.}

	The determination of a random source by the corresponding passive measurement was also recently studied in \cite{bao2016inverse,Lu1,Yuan1}, and the determination of a random potential by the corresponding active measurement was established in \cite{caro2016inverse}. We also refer to \cite{LassasA} and the references therein for more relevant studies on the determination of a random potential. The simultaneous recovery of an unknown source and its surrounding potential was also investigated in the literature. In \cite{KM1,liu2015determining}, motivated by applications in thermo- and photo-acoustic tomography, the simultaneous recovery of an unknown source and its surrounding medium parameter was considered. The simultaneous recovery study in \cite{KM1,liu2015determining} was confined to the deterministic setting and associated mainly with the passive measurement.

	In this paper, we consider the 
	recovery of an unknown random source and an unknown potential term associated with the Schr\"odinger system \eqref{eq:1}. 
	The major novelty of our unique recovery results compared to those existing ones in the literature is that on the one hand, both the random source and the potential are unknown, and on the other hand, we use both passive and active measurements for the unique recovery. We established three unique recovery results. 
	
	
	\begin{thm} \label{thm:Unisigma-SchroEqu2018}
		\sq{Without knowing $V$ and $f$ in system \eqref{eq:1}, the data set $\mathcal M_1$ can recover $\sigma^2$ almost surely.}
	\end{thm}
	
	\begin{rem}
		Theorem~\ref{thm:Unisigma-SchroEqu2018} implies that the variance function can be uniquely recovered without \emph{a priori} knowledge of $f$ or $V$. 
		\sq{Moreover, since the passive measurement $\mathcal M_1$ is used, Theorem \ref{thm:Unisigma-SchroEqu2018} indicates that the variance function can be uniquely recovered by a single realization of the passive scattering measurement. 
		Moreover, for the sake of simplicity, we set the wave number $k$ in the definition of $\mathcal M_1$ to be running over all positive real numbers. But in practice, it is enough to let $k$ be greater than any fixed positive number. This remark equally applies to Theorem \ref{thm:UniPot1-SchroEqu2018}.}
	\end{rem}
	
	
	\begin{thm} \label{thm:UniPot1-SchroEqu2018}
		\sq{Without knowing $\sigma$ and $f$ in system \eqref{eq:1}, the data set $\mathcal M_2$ uniquely recovers the potential $V$.}
	\end{thm}
	
	\begin{rem}
		Theorem \ref{thm:UniPot1-SchroEqu2018} shows that the potential $V$ can be uniquely recovered without knowing the random source, namely $\sigma$ and $f$. Moreover, we only make use of 
		a single realization of the active scattering measurement. 
	\end{rem}

	\begin{thm} \label{thm:UniSou1-SchroEqu2018}
		\sq{In system \eqref{eq:1}, suppose that $\sigma$ is unknown and the potential $V$ is known in advance. Then there exists a positive constant $C$ that depends only on $D$ such that if $\nrm[L^\infty(\R^3)]{V} < C$, 
		the data set $\mathcal M_3$ can uniquely recover the expectation $f$.}
	\end{thm}
	
	\jz{The rest of the paper is outlined as follows. In Section \ref{sec:MADP-SchroEqu2018}, we present the mathematical analysis of the forward scattering problem given in \eqref{eq:1}. 
	Section \ref{sec:AsyEst-SchroEqu2018}  establishes some asymptotic estimates, which are of key importance in the recovery of the variance function. 
	In Section \ref{sec:RecVar-SchroEqu2018}, we prove the first recovery result of the variance function with a single realization of the passive scattering measurement. 
	Section \ref{sec:RecPS-SchroEqu2018} is devoted to the second and third recovery results of the potential and the random source. We conclude the work with some discussions in \mbox{Section \ref{sec:Conclusions-SchroEqu2018}.}}
	
	\sq{
	\section{Mathematical analysis of the direct problem} \label{sec:MADP-SchroEqu2018}}

	\sq{In this section, the uniqueness and existence of a {\it mild solution} is established for the system \eqref{eq:1}. Before analyzing the direct problem, some preparations are made in the beginning.
	In Section \ref{subsec:NotandAss-SchroEqu2018}, we introduce some preliminaries which are used throughout the rest of the paper. 
	Some technical lemmas that are necessary for the analysis of both the direct and inverse problems are presented in Section \ref{subsec:STLemmas-SchroEqu2018}. 
	In Section \ref{subset:WellDefined-SchroEqu2018}, we give the well-posedness of the direct problem.}

	\subsection{Preliminaries} \label{subsec:NotandAss-SchroEqu2018}
	\sq{Let us first introduce the generalized Gaussian white noise $\dot B_x(\omega)$ \cite{kusuoka1982support}.} 
	To give a brief introduction, we write $\dot B_x(\omega)$ temporarily as $\dot B(x,\omega)$.
	It is known that $\dot B(\cdot,\omega) \in H_{loc}^{-3/2-\epsilon}(\R^3)$ almost surely for any $\epsilon\in\mathbb{R}_+$ \cite{kusuoka1982support}. 
	Then $\dot B \colon \omega \in \Omega \mapsto \dot B(\cdot,\omega) \in \scrD'(D)$ defines a map from the probability space to the space of the generalized functions. 
	Here, $\scrD(D)$ signifies the space consisting of smooth functions that are compactly supported in $D$, and $\scrD'(D)$ signifies its dual space. 
	For any $\varphi \in \scrD(D)$, $\dot B \colon \omega \in \Omega \mapsto \agl[\dot B(x,\omega), \varphi(x)] \in \R$ is assumed to be a Gaussian random variable with zero-mean and $\int_{D} |\varphi(x)|^2 \dif{x}$ as its variance. 
	We also recall that a function $\psi$ in $L_{loc}^1(\R^n)$ defines a distribution through ${\agl[\psi,\varphi] = \int_{\R^n} \psi(x) \varphi(x) \dif{x}}$ \cite{caro2016inverse}. Then $\dot B(x,\omega)$ satisfies:
	\[
	\agl[\dot B(\cdot,\omega), \varphi(\cdot)] \sim \mathcal{N}(0,\nrm[L^2(D)]{\varphi}^2), \quad \forall \varphi \in \scrD(D).
	\]
	Moreover, the covariance of the $\dot B(x,\omega)$ is assumed to satisfy the following property. For every $\varphi$, $\psi$ in $\scrD(D)$, the covariance between $\agl[\dot B(\cdot,\omega), \varphi]$ and $\agl[\dot B(\cdot,\omega), \psi]$ is defined as $\int_{D} \varphi(x) \psi(x) \dif{x}$:
	\begin{equation} \label{eq:ItoIso-SchroEqu2018}
	\mathbb{E} \big( \agl[\dot B(\cdot,\omega), \varphi] \agl[\dot B(\cdot,\omega), \psi] \big) := \int_{D} \varphi(x) \psi(x) \dif{x}.
	\end{equation}
	These aforementioned definitions can be generalized to the case where $\varphi, \psi \in L^2(D)$ by the density arguments. 
	The $\delta(x) \dot B(x,\omega)$ is defined as
	\begin{equation} \label{eq:sigmaB-SchroEqu2018}
	\delta(x) \dot B(x,\omega) \colon \varphi \in L^2(D) \mapsto \agl[\dot B(\cdot,\omega), \delta(\cdot)\varphi(\cdot)] \in \R.
	\end{equation}

	Secondly, let's set
	$$\Phi(x,y) = \Phi_k(x,y) := \frac {e^{ik|x-y|}}{4\pi|x-y|}, \quad x\in\mathbb{R}^3\backslash\{y\}.$$
	$\Phi_k$ is the outgoing fundamental solution, centered at $y$, to the differential operator $-\Delta-k^2$. Define the resolvent operator $\Rk$,
	\begin{equation} \label{eq:DefnRk-SchroEqu2018}
	\Rk(\varphi)(x) = (\Rk \varphi)(x) := \int_{\supp \varphi} \Phi_k(x,y) \varphi(y) \dif{y}, \quad x \in \R^3,
	\end{equation}
	where $\varphi$ can be any measurable function on $\mathbb{R}^3$ as long as the \eqref{eq:DefnRk-SchroEqu2018} is well-defined for almost all $x$ in $\R^3$. Similar to the \eqref{eq:DefnRk-SchroEqu2018}, we define $\Rk(\delta \dot{B}_x)(\omega)$ as
	\begin{equation} \label{eq:RkSigmaBDefn-SchroEqu2018}
	\Rk(\delta \dot{B}_x)(\omega) := \agl[\dot B(\cdot,\omega), \delta(\cdot) \Phi(x,\cdot)],
	\end{equation}
	for any $\delta \in L^{\infty}(\R^3)$ with $\supp \delta \subseteq D$. We write $\Rk(\delta \dot{B}_x)(\omega)$ as $\Rk(\delta \dot{B}_x)$ for short. 
	We may also write $\Rk(\delta \dot{B}_x)$ as $\int_{\R^3} \Phi_k(x,y) \delta(y) \dot B_y \dif{y}$ or $\int_{\R^3} \Phi_k(x,y) \delta(y) \dif{B_y}$. 
	We may omit the subscript $x$ in $\Rk(\delta \dot B_x)$ if it is clear in the context.
	
	Write $\agl[x] := (1+|x|^2)^{1/2}$ for $x \in \R^3$. We introduce the following weighted $L^2$-norm and the corresponding function space over $\R^3$ for any $s \in \R$,
	\begin{equation} \label{eq:WetdSpace-SchroEqu2018}
	\left\{
	\begin{aligned}
	\nrm[L_{s}^2(\R^3)]{f} & := \nrm[L^2(\R^3)]{\agl[\cdot]^{s} f(\cdot)} = \Big( \int_{\R^3} \agl[x]^{2s} |f|^2 \dif{x} \Big)^{\frac 1 2}, \\
	L_{s}^2(\R^3) & := \left\{ f\in L_{loc}^1(\mathbb{R}^3); \nrm[L_{s}^2(\R^3)]{f} < +\infty \right\}.
	\end{aligned}\right.
	\end{equation}
	We also define $L_{s}^2(S)$ for any measurable subset $S$ in $\R^3$ by replacing $\R^3$ in \eqref{eq:WetdSpace-SchroEqu2018} with $S$. In what follows, we may denote $L_s^2(\R^3)$ as $L_s^2$ for short if without ambiguities.
	
	\jz{In the sequel, we write $\mathcal{L}(\mathcal A, \mathcal B)$ to denote the set of all the linear bounded mappings from a norm vector space $\mathcal A$ to a norm vector space $\mathcal B$. 
	For any mapping $\mathcal K \in \mathcal{L}(\mathcal A, \mathcal B)$, we denote its operator norm as $\nrm[\mathcal{L}(\mathcal A, \mathcal B)]{\mathcal K}$. 
	We write the identity operator as $I$.
	We also use notations $C$ and its variants, such as $C_D$ and $C_{D,f}$ to represent some generic constant(s) whose particular definition may change line by line. 
	We use $\mathcal{A}\lesssim \mathcal{B}$ to signify $\mathcal{A}\leq C \mathcal{B}$ and $\mathcal{A} \simeq \mathcal{B}$ to signify $\mathcal{A} = C \mathcal{B}$, for some generic positive constant $C$. We denote ``almost everywhere'' as~``a.e.''~and ``almost surely'' as~``a.s.''~for short. 
	We use $|\mathcal S|$ to denote the Lebesgue measure of any Lebesgue-measurable set $\mathcal S$. Define $M(x) = \sup_{y \in D}|x-y|$, and $\text{diam}\,D := \sup_{x,y \in D} |x-y|$, where $D$ is the bounded domain containing $\supp \sigma$, $\supp V$, $\supp f$ and the origin. 
	Thus we have
	$M(0) \leq \text{diam}\,D < \infty$. It can be verified that
		\begin{equation} \label{eq:Contain-SchroEqu2018}
		\{ y-x \in \R^3 ; |x| \leq 2 M(0), y \in D \} \subseteq \{ z \in \R^3 ; |z| \leq 3\,\text{diam}\,D \}.
		\end{equation}}
	This is because $|y-x| \leq |y| + |x| \leq \text{diam}\,D + 2M(0) \leq 3\text{diam}\,D$.
	
	\subsection{Several technical lemmas} \label{subsec:STLemmas-SchroEqu2018}
	Several important technical lemmas are presented here. 
	
	\begin{lem} \label{lemma:RkBoundedR3-SchroEqu2018}
		For any $\varphi \in L^\infty(\R^3)$ with $\supp \varphi \subseteq D$ and any $\epsilon \in \mathbb{R}_+$, we have
		\[
		\Rk \varphi \in L_{-1/2-\epsilon}^2.
		\]
		
	\end{lem}
	\begin{proof}[Proof of Lemma \ref{lemma:RkBoundedR3-SchroEqu2018}]
		Assume that $\varphi$ belongs to $L^\infty(\R^3)$ with its support contained in $D$. \jz{Obviously we have that $\nrm[L^2(D)]{\varphi} < +\infty$. Using the Cauchy-Schwarz inequality we have
			\begin{align} \label{eq:Rk1-SchroEqu2018}
			\nrm[L_{-1/2-\epsilon}^2]{\Rk \varphi}^2
			& \lesssim \int_{\R^3} \agl[x]^{-1-2\epsilon} \big( \int_D \frac{1}{|x-y|^2} \dif{y} \big) \cdot \big( \int_D |\varphi(y)|^2 \dif{y} \big) \dif{x} \nonumber\\
			& \lesssim \nrm[L^2(D)]{\varphi}^2 \Big[ \int_{|x| \leq 2M(0)} \big( \int_D \frac{1}{|x-y|^2} \dif{y} \big) \dif{x} \nonumber\\
			& \quad\quad + \int_{|x| > 2M(0)} \agl[x]^{-1-2\epsilon} \agl[x]^{-2} \dif{x} \Big].
			\end{align}
			By the change of variable, the first term in the square brackets in \eqref{eq:Rk1-SchroEqu2018} satisfies
			\begin{equation} \label{eq:Rk2-SchroEqu2018}
			\int_{|x| \leq 2M(0)} \big( \int_D \frac{1}{|x-y|^2} \dif{y} \big) \dif{x} = \int_{|x| \leq 2M(0)} \big( \int_{z \in \{ y-x \,;\, y \in D \}} \frac{1}{|z|^2} \dif{z} \big) \dif{x}.
			\end{equation}
			From \eqref{eq:Contain-SchroEqu2018}, we can continue \eqref{eq:Rk2-SchroEqu2018} as
			\begin{align}
			\int_{|x| \leq 2M(0)} \big( \int_D \frac{1}{|x-y|^2} \dif{y} \big) \dif{x}
			& \leq \int_{|x| \leq 2M(0)} \big( \int_{\{z \,;\, |z| \leq 3\,\text{diam}\,D \}} \frac{1}{|z|^2} \dif{z} \big) \dif{x} \nonumber\\
			& = \int_{|x| \leq 2M(0)} \big( 12\pi\,\text{diam}\,D \big) \dif{x} < +\infty. \label{eq:Rk3-SchroEqu2018}
			\end{align}
			Meanwhile, the second term in the square brackets in \eqref{eq:Rk1-SchroEqu2018} satisfies
			\begin{equation} \label{eq:Rk4-SchroEqu2018}
			\int_{|x| > 2M(0)} \agl[x]^{-1-2\epsilon} \agl[x]^{-2} \dif{x} \leq \int_{\R^3} \agl[x]^{-3-2\epsilon} \dif{x} < +\infty.
			\end{equation}
			Note that \eqref{eq:Rk4-SchroEqu2018} holds for every $\epsilon \in \mathbb{R}_+$. Combining \eqref{eq:Rk1-SchroEqu2018}, \eqref{eq:Rk3-SchroEqu2018} and \eqref{eq:Rk4-SchroEqu2018}, we conclude
			\[
			\nrm[L_{-1/2-\epsilon}^2]{\Rk \varphi}^2 < +\infty.
			\]
		}
		The proof is complete.
	\end{proof}
	
	\jz{Now we present a special version of Agmon's estimates for the convenience of our reader (cf. \cite{eskin2011lectures}). This special version will be used when proving Lemma \ref{lemma:RkVBoundedR3-SchroEqu2018}.}
	\begin{lem}[Agmon's estimates \cite{eskin2011lectures}] \label{lemma:AgmonEst-SchroEqu2018}
		For any $\epsilon > 0$, \jz{there exists some $k_0 \geq 2$ such that for any $k > k_0$} we have
		\begin{equation} \label{eq:AgmonEst-SchroEqu2018}
		\nrm[L_{-1/2-\epsilon}^2]{\Rk \varphi} \leq C_\epsilon k^{-1} \nrm[L_{1/2+\epsilon}^2]{\varphi}, \quad \forall \varphi \in L_{1/2+\epsilon}^2
		\end{equation}
		where $C_\epsilon$ is independent of $k$ {and $\varphi$}.
	\end{lem}
	
	\sq{The proof of Lemma \ref{lemma:AgmonEst-SchroEqu2018} can be found in \cite{eskin2011lectures}. 
	The symbol $k_0$ appearing in Lemma \ref{lemma:AgmonEst-SchroEqu2018} is preserved for future use.}

	\begin{lem} \label{lemma:RkVBoundedR3-SchroEqu2018}
		For any fixed $\epsilon \geq 0$, when $k > k_0$, we have
		$$\nrm[\mathcal{L}(L_{-1/2-\epsilon}^2, L_{-1/2-\epsilon}^2)]{\Rk \circ V} \leq C_{\epsilon,D,V} k^{-1},$$
		where the constant $C_{\epsilon,D,V}$ depends on $\epsilon, D$ and $V$ but is independent of $k$.
	\end{lem}
	\begin{proof}[Proof of Lemma \ref{lemma:RkVBoundedR3-SchroEqu2018}]
		By Lemma \ref{lemma:AgmonEst-SchroEqu2018}, when $k > k_0$, we have the following estimate,
		$$\nrm[L_{-1/2-\epsilon}^2]{\Rk V u} = \nrm[L_{-1/2-\epsilon}^2]{\Rk (Vu)} \leq C_\epsilon k^{-1} \nrm[L_{1/2+\epsilon}^2]{Vu}.$$
		Due to the boundedness of $\supp V$, there holds $\nrm[L_{1/2+\epsilon}^2]{Vu} \leq C_{D,V} \nrm[L_{-1/2-\epsilon}^2]{u}$ for some constant $C_{D,V}$ depending on $D$ and $V$ but independent of $u$ and $\epsilon$. Thus, we have
		$$\nrm[L_{-1/2-\epsilon}^2]{\Rk Vu} \leq C_{\epsilon,D,V} k^{-1} \nrm[L_{-1/2-\epsilon}^2]{u}.$$
		
		The proof is complete. 
	\end{proof}
	
	\medskip
	
	\sq{In the rest of the paper, we use $k^*$ to represent the maximum between the quantity $k_0$ 
	originated from Lemma \ref{lemma:AgmonEst-SchroEqu2018}
	and the quantity
	\[
	\sup_{k \in \R_+} \{ k \,;\, \nrm[\mathcal{L}{(L_{-1/2-\epsilon}^2, L_{-1/2-\epsilon}^2)}]{\mathcal{R}_k V} \geq 1 \} \ + \ 1.
	\]
	This choice of $k^*$ guarantees that if $k \geq k^*$, both the inequality \eqref{eq:AgmonEst-SchroEqu2018} and the Neumann expansion
	\(
	{(I - \mathcal{R}_k V)^{-1} \ = \ \sum_{j \geq 0} (\mathcal{R}_k V)^j}
	\)
	in $L_{-1/2-\epsilon}^2$ hold. 
	
	For the subsequent analysis we also need a local version of Lemma \ref{lemma:RkVBoundedR3-SchroEqu2018}.}
	
	\begin{lem} \label{lemma:RkVBounded-SchroEqu2018}
		When $k > k_0$, we have
		\begin{equation} \label{eq:RkVbdd-SchroEqu2018}
		\nrm[\mathcal{L}(L^2(D), L^2(D))]{\Rk V} \leq C_{D,V} k^{-1},
		\end{equation}
		for some constant $C_{D,V}$ depending on $D$ and $V$ but independent of $k$. Moreover, for every $\varphi \in L^2(\R^3)$ with $\supp \varphi \subseteq D$, then
		\begin{equation} \label{eq:VRkbdd-SchroEqu2018}
		\nrm[L^2(D)]{V\Rk \varphi} \leq C_{D,V} k^{-1} \nrm[L^2(D)]{\varphi},
		\end{equation}
		for some constant $C_{D,V}$ depending on $D$ and $V$ but independent of $\varphi$ and $k$.
	\end{lem}
	\begin{proof}
		\jz{For any $\varphi \in L^2(D)$, thanks to the boundedness of $D$ we have
			\begin{equation} \label{eq:RkVbddInter1-SchroEqu2018}
			\nrm[L^2(D)]{\Rk V\varphi} \leq C_D \nrm[L_{-1}^2]{\Rk (V\varphi)}.
			\end{equation}
			By Lemma \ref{lemma:AgmonEst-SchroEqu2018} (letting the $\epsilon$ in Lemma \ref{lemma:AgmonEst-SchroEqu2018} be $\frac 1 2$), we conclude that
			\begin{equation} \label{eq:RkVbddInter2-SchroEqu2018}
			\nrm[L_{-1}^2]{\Rk (V\varphi)} \leq C k^{-1} \nrm[L_1^2]{V\varphi}.
			\end{equation}
			By virtue of the boundedness of $V$, we have
			\begin{equation} \label{eq:RkVbddInter3-SchroEqu2018}
			\nrm[L_1^2]{V\varphi} \leq C_{D,V}\nrm[L^2(D)]{\varphi}.
			\end{equation}
			Combining \eqref{eq:RkVbddInter1-SchroEqu2018}-\eqref{eq:RkVbddInter3-SchroEqu2018}, we arrive at \eqref{eq:RkVbdd-SchroEqu2018}.
			
			
			To prove \eqref{eq:VRkbdd-SchroEqu2018}, by Lemma \ref{lemma:AgmonEst-SchroEqu2018}, we have
			$$\nrm[L^2(D)]{\Rk \varphi} \leq C_D \nrm[L_{-1}^2]{\Rk \varphi} \leq C_D k^{-1} \nrm[L_1^2]{\varphi} \leq C_D k^{-1} \nrm[L^2(D)]{\varphi}.$$
			Therefore,
			$$\nrm[L^2(D)]{V\Rk \varphi} \leq \nrm[L^\infty(D)]{V} \cdot \nrm[L^2(D)]{\Rk \varphi} \leq C_{D,V} k^{-1} \nrm[L^2(D)]{\varphi}.$$
			
			The proof is complete.}
	\end{proof}
	
	\medskip
	
	Lemma \ref{lemma:RkSigmaB-SchroEqu2018} shows some basic properties of $\Rk(\sigma \dot{B}_x)$ defined in \eqref{eq:RkSigmaBDefn-SchroEqu2018}.
	\begin{lem} \label{lemma:RkSigmaB-SchroEqu2018}
		We have
		$$\Rk(\sigma \dot B_x) \in L_{-1/2-\epsilon}^2 \quad \ass.$$
		Moreover, we have
		\begin{equation*}
		\mathbb{E} \nrm[L^2(D)]{\Rk(\sigma \dot B_x)} < C < +\infty
		\end{equation*}
		for some constant $C$ independent of $k$.
	\end{lem}
	
	\begin{proof}
		From \eqref{eq:RkSigmaBDefn-SchroEqu2018}, \eqref{eq:sigmaB-SchroEqu2018} and \eqref{eq:ItoIso-SchroEqu2018}, one can compute,
		\begin{align*}
		\mathbb{E} ( \nrm[L_{-1/2-\epsilon}^2]{\Rk(\sigma \dot{B}_x)}^2 )
		& = \int_{\R^3} \agl[x]^{-1-2\epsilon} \mathbb{E} \big( \agl[\dot B(\cdot,\omega), \sigma(\cdot) \Phi(x,\cdot)] \agl[\dot B(\cdot,\omega), \sigma(\cdot) \overline{\Phi}(x,\cdot)] \big) \dif{x} \\
		& = \int_{\R^3} \agl[x]^{-1-2\epsilon} \int_D \sigma^2(y) \frac 1 {16\pi^2 |x-y|^2} \dif{y} \dif{x} \\
		& \leq C \nrm[L^\infty(D)]{\sigma}^2 \int_{\R^3} \agl[x]^{-1-2\epsilon} \int_D |x-y|^{-2} \dif{y} \dif{x}.
		\end{align*}
		By arguments similar to the ones used in the proof of Lemma \ref{lemma:RkBoundedR3-SchroEqu2018} we arrive at
		\begin{equation} \label{eq:Rksigma2Bounded-SchroEqu2018}
		\mathbb{E} (\nrm[L_{-1/2-\epsilon}^2]{\Rk(\sigma \dot{B}_x)}^2) \leq C_D < +\infty,
		\end{equation}
		for some constant $C_D$ depending on $D$ but not on $k$. By the H\"older inequality applied to the probability measure, (\ref{eq:Rksigma2Bounded-SchroEqu2018}) gives
		\begin{equation} \label{eq:RksigmaBoundedCD-SchroEqu2018}
		\mathbb{E} (\nrm[L_{-1/2-\epsilon}^2]{\Rk(\sigma \dot{B}_x)}) \leq [ \mathbb{E} ( \nrm[L_{-1/2-\epsilon}^2]{\Rk(\sigma \dot{B}_x)}^2 ) ]^{1/2} \leq \sq{C_D^{1/2}} < +\infty,
		\end{equation}
		for some constant $C_D$ independent of $k$. The inequality (\ref{eq:RksigmaBoundedCD-SchroEqu2018}) gives
		$$\Rk(\sigma \dot B_x) \in L_{-1/2-\epsilon}^2 \quad \ass.$$
		By replacing $\R^3$ with $D$ and deleting all the terms $\agl[x]^{-1-2\epsilon}$ in the derivations above, one arrives at $\mathbb{E} \nrm[L^2(D)]{\Rk(\sigma \dot{B}_x)} < +\infty$. 
		The proof is done.
	\end{proof}
	
	\subsection{The well-posedness of the \textbf{DP}} \label{subset:WellDefined-SchroEqu2018}
	
	For a particular realization of the random sample $\omega \in \Omega$, the term $\dot B_x(\omega)$,
	\jz{treated as a function of the spatial argument $x$, could be very rough. 
	The roughness of this term could make these classical second-order elliptic PDEs theories invalid to \eqref{eq:1}.} 
	Due to this reason, the notion of the {\em mild solution} is introduced for random PDEs  (cf. \cite{bao2016inverse}). 
	In what follows, we adopt the mild solution in our problem setting, and we show that this mild solution and the corresponding far-field pattern are well-posed in a proper sense.
	
	Reformulating \eqref{eq:1} into the Lippmann-Schwinger equation formally (cf. \cite{colton2012inverse}), we have
	\begin{equation} \label{eq:LippSch-SchroEqu2018}
	(I - \Rk V) u = \alpha \cdot u^i - \Rk f - \Rk(\sigma \dot B_x),
	\end{equation}
	where the term $\Rk(\sigma \dot B_x)$ is defined by (\ref{eq:RkSigmaBDefn-SchroEqu2018}). Recall that $u^{sc} = u - \alpha \cdot u^i$. From (\ref{eq:LippSch-SchroEqu2018}) we have
	\begin{equation} \label{eq:uscDefn-SchroEqu2018}
	(I - \Rk V) u^{sc} = \alpha \Rk V u^i - \Rk f - \Rk(\sigma \dot B_x).
	\end{equation}
	
	\begin{thm} \label{thm:MildSolUnique-SchroEqu2018}
		\sq{When $k > k^*$, there exists a unique stochastic process $u^{sc}(\cdot,\omega) \colon \R^3 \to \mathbb C$ such that $u^{sc}(x)$ satisfies \eqref{eq:uscDefn-SchroEqu2018} a.s.\,, and ${u^{sc}(\cdot,\omega) \in L_{-1/2-\epsilon}^2 \ass}$for any $\epsilon\in\mathbb{R}_+$. 
		Moreover, we have
		\begin{equation} \label{thm:SolPosed-SchroEqu2018}
			\nrm[L_{-1/2-\epsilon}^2]{u^{sc}(\cdot,\omega)} 
			\lesssim 
			\nrm[L_{1/2+\epsilon}^2]{\alpha V u^i}
			+
			\nrm[L_{1/2+\epsilon}^2]{f}
			+
			\nrm[L_{-1/2-\epsilon}^2]{\Rk(\sigma \dot B_x)}.
		\end{equation}
		Then we call $u(x) := u^{sc} + \alpha \cdot u^i(x)$ the {\em mild solution} to the random scattering system \eqref{eq:1}.}
	\end{thm}
	
	\begin{proof}
		\sq{By Lemmas \ref{lemma:RkBoundedR3-SchroEqu2018}, \ref{lemma:RkVBoundedR3-SchroEqu2018} and \ref{lemma:RkSigmaB-SchroEqu2018}, we see
		$$F := \alpha \Rk V u^i - \Rk f - \Rk(\sigma \dot B_x) \in L_{-1/2-\epsilon}^2.$$
		Note that $k > k^*$, so the term $\sum_{j=0}^\infty (\Rk V)^j$ is well-defined, thus the term $\sum_{j=0}^\infty (\Rk V)^j F$ belongs to $L_{-1/2-\epsilon}^2$. Because $\sum_{j=0}^\infty (\Rk V)^j = (I - \Rk V)^{-1}$, we see $(I - \Rk V)^{-1} F \in L_{-1/2-\epsilon}^2$. Let $u^{sc} := (I - \Rk V)^{-1} F \in L_{-1/2-\epsilon}^2$, then $u^{sc}$ is the unique solution of \eqref{eq:uscDefn-SchroEqu2018}. That is, the existence  of the mild solution is proved The uniqueness of the mild solution follows from the invertibility of the operator $(I - \Rk V)^{-1}$. 
		
		From \eqref{eq:uscDefn-SchroEqu2018} and Lemmas \ref{lemma:AgmonEst-SchroEqu2018}-\ref{lemma:RkVBoundedR3-SchroEqu2018}, we have
		\begin{align*}
		\nrm[L_{-1/2-\epsilon}^2]{u^{sc}(\cdot,\omega)}
		& = \nrm[L_{-1/2-\epsilon}^2]{(I - \Rk V)^{-1} (\alpha \Rk V u^i - \Rk f - \Rk(\sigma \dot B_x))} \\
		& \leq \sum_{j \geq 0} \nrm[\mathcal L(L_{-1/2-\epsilon}^2,L_{-1/2-\epsilon}^2)]{\Rk V}^j \cdot \nrm[L_{-1/2-\epsilon}^2]{\alpha \Rk V u^i - \Rk f - \Rk(\sigma \dot B_x)} \\
		& \leq C (
		\nrm[L_{-1/2-\epsilon}^2]{\alpha \Rk V u^i}
		+
		\nrm[L_{-1/2-\epsilon}^2]{\Rk f}
		+
		\nrm[L_{-1/2-\epsilon}^2]{\Rk(\sigma \dot B_x)}) \\
		& \leq C (
		\nrm[L_{1/2+\epsilon}^2]{\alpha V u^i}
		+
		\nrm[L_{1/2+\epsilon}^2]{f}
		+
		\nrm[L_{-1/2-\epsilon}^2]{\Rk(\sigma \dot B_x)}
		).
		\end{align*}
		Therefore \eqref{thm:SolPosed-SchroEqu2018} is proved. 
		The proof is complete.}
	\end{proof}

	\bigskip

	Next we show that the far-field pattern is well-defined in the $L^2$ sense. From \eqref{eq:uscDefn-SchroEqu2018} we derive that
	\begin{align*}
	u^{sc}
	& = (I - \Rk V)^{-1} \big( \alpha \Rk V u^i - \Rk f - \Rk(\sigma \dot B_x) \big) \\
	& = \Rk (I - V \Rk)^{-1} (\alpha V u^i - f - \sigma \dot B_x).
	\end{align*}
	Therefore, we define the far-field pattern of the scattered wave $u^{sc}(x,k,d,\omega)$ formally in the following manner,
	\begin{equation} \label{eq:uInftyDefn-SchroEqu2018}
	u^\infty(\hat x,k,d,\omega) := \frac 1 {4\pi} \int_D e^{-ik\hat x \cdot y} (I - V \Rk)^{-1} (\alpha V u^i - f - \sigma \dot B_y) \dif{y}, \quad \hat x \in \mathbb{S}^2.
	\end{equation}
	
	The another result concerning the \textbf{DP} is Theorem \ref{thm:FarFieldWellDefined-SchroEqu2018}, showing that $u^\infty(\hat x,k,d,\omega)$ is well-defined.
	\begin{thm} \label{thm:FarFieldWellDefined-SchroEqu2018}
		Define the far-field pattern of the mild solution as in \eqref{eq:uInftyDefn-SchroEqu2018}. When \jz{$k > k^*$}, there is a subset \jz{$\Omega_0 \subset \Omega$} with zero measure $\mathbb P (\Omega_0) = 0$, such that there holds
		\[
		u^\infty(\hat x,k,d,\omega) \in L^2(\mathbb{S}^2),\ \  \Forall \omega \in \Omega \backslash \Omega_0.
		\]
	\end{thm}
	\begin{proof}[Proof of Theorem \ref{thm:FarFieldWellDefined-SchroEqu2018}]
		\jz{By Lemma \ref{lemma:RkVBounded-SchroEqu2018}, 
		$$\nrm[\mathcal{L}(L^2(D), L^2(D))]{V \Rk} \leq C k^{-1} < 1$$
		when $k$ is sufficiently large. Therefore we have,
		\begin{align}
		|u^\infty(\hat x)|^2
		& \lesssim |D|^2 \cdot \int_D |\sum_{j \geq 0} (V \Rk)^j (\alpha V u^i - f)|^2 \dif{y} \nonumber\\
		& \ \ \ \ + \big| \int_D e^{-ik\hat x \cdot y} \sum_{j \geq 1} (V \Rk)^j (\sigma \dot B_y) \dif{y} \big|^2 \nonumber\\
		& \ \ \ \ + \big| \int_D e^{-ik\hat x \cdot y} \sigma \dot B_y \dif{y} \big|^2 \nonumber\\
		& =: f_1(\hat x, k) + f_2(\hat x, k,\omega) + f_3(\hat x, k,\omega). \label{eq:a1}
		\end{align}
		We next derive estimates on each term $f_j ~(j=1,2,3)$ defined in \eqref{eq:a1}. For $f_1$, we have
		\begin{align}
		f_1(\hat x, k)
		& \leq C |D|^2 \cdot ( \sum_{j \geq 0} k^{-j} \nrm[L^2(D)]{\alpha V u^i - f} )^2 \leq C |D|^2 (\nrm[L^2(D)]{V} + \nrm[L^2(D)]{f})^2. \label{eq:f1-SchroEqu2018}
		\end{align}
		For $f_2$, by utilizing \eqref{eq:VRkbdd-SchroEqu2018}, one can compute
		\begin{equation} \label{eq:f2Inter-SchroEqu2018}
		f_2(\hat x, k, \omega)
		\leq C \int_D |\sum_{j \geq 0} (V \Rk)^j V \Rk(\sigma \dot B_y)|^2 \dif{y} \leq C \big( \sum_{j \geq 0} k^{-j} \nrm[L^2(D)]{V \Rk(\sigma \dot B_y)} \big)^2.
		\end{equation}
		By virtue of the boundedness of the support of $V$, we can continue \eqref{eq:f2Inter-SchroEqu2018} as
		\begin{equation} \label{eq:f2-SchroEqu2018}
		f_2(\hat x, k, \omega) \leq C \big( \sum_{j \geq 0} k^{-j} \nrm[L_{-1/2-\epsilon}^2]{V \Rk(\sigma \dot B_y)} \big)^2 \leq C_V \nrm[L_{-1/2-\epsilon}^2]{\Rk(\sigma \dot B_y)}^2.
		\end{equation}}
		By (\ref{eq:ItoIso-SchroEqu2018}), the expectation of $f_3(\hat x, k, \omega)$ is
		\begin{equation} \label{eq:f3-SchroEqu2018}
		\mathbb{E} f_3(\hat x, k, \omega) = \mathbb{E} |\agl[\dot B_y, e^{-ik\hat x \cdot y} \sigma(y)]|^2 = \int_D |\sigma(y)|^2 \dif{y}.
		\end{equation}
		Combining \eqref{eq:Rksigma2Bounded-SchroEqu2018}, \eqref{eq:a1}-\eqref{eq:f1-SchroEqu2018} and \eqref{eq:f2-SchroEqu2018}-\eqref{eq:f3-SchroEqu2018}, we arrive at
		\begin{align}
		\mathbb E |u^\infty(\hat x)|^2
		& \leq C |D|^2 (\nrm[L^2(D)]{V} + \nrm[L^2(D)]{f})^2 + C_V \mathbb E (\nrm[L_{-1/2-\epsilon}^2]{\Rk(\sigma \dot B_y)}^2) + \int_D |\sigma(y)|^2 \dif{y} \nonumber\\
		& \leq C < +\infty \label{eq:a2-SchroEqu2018}
		\end{align}
		for some positive constant $C$. From \eqref{eq:a2-SchroEqu2018} we arrive at
		\begin{equation} \label{eq:a3-SchroEqu2018}
		\mathbb E \int_{\mathbb{S}^2} |u^\infty(\hat x)|^2 \dif{S} \leq C < +\infty.
		\end{equation}
		Our conclusion follows from \eqref{eq:a3-SchroEqu2018} immediately.
	\end{proof}
	

	\sq{\section{Some asymptotic estimates} \label{sec:AsyEst-SchroEqu2018}}
		
	\sq{This section is devoted to some preparations of the recovery of the variance function. To recovery $\sigma^2(x)$, only the passive far-field patterns are utilized. 
	Therefore, throughout this section, the $\alpha$ in \eqref{eq:1} is set to be 0. Motivated by \cite{caro2016inverse}, our recovery formula of the variance function is of the form
	\begin{equation} \label{eq:example-SchroEqu2018}
	\frac 1 {K} \int_{K}^{2K} \overline{u^\infty(\hat x,k,\omega)} \cdot u^\infty(\hat x,k+\tau,\omega) \dif{k}.
	\end{equation}}
	\sq{After expanding $u^\infty(\hat x,k,\omega)$ in the form of Neumann series, there will be several crossover terms in \eqref{eq:example-SchroEqu2018} which decay in different rates in terms of $K$. In this section, we focus on the asymptotic estimates of these terms, which pave the way to the recovery of $\sigma^2(x)$. The recovery of $\sigma^2(x)$ is presented in the next section.}

	\medskip

	\sq{To start out, we write
	\begin{equation} \label{eq:u1-SchroEqu2018}
	u_1^\infty(\hat x,k,\omega) := u^\infty(\hat x,k,\omega) - \mathbb{E} u^\infty(\hat x,k).
	\end{equation}
	Note that $u_1^\infty$ is independent of the incident direction $d$. Assume that $k > k^*$, then the operator $(I - \Rk V)^{-1}$ has the Neumann expansion $\sum_{j=0}^{+\infty} (\Rk V)^j$. By \eqref{eq:uInftyDefn-SchroEqu2018} and \eqref{eq:u1-SchroEqu2018} we have
	\begin{align}
	u_1^\infty(\hat x,k,\omega)
	\ = & \ \frac {-1} {4\pi} \sum_{j=0}^{+\infty} \int_D e^{-ik\hat x \cdot y} (\Rk V)^j (\sigma \dot B_y) \dif{y}, \quad \hat x \in \mathbb{S}^2 \nonumber\\
	:= & \ \frac {-1} {4\pi} \big[ F_0(k,\hat x) + F_1(k,\hat x) \big], \label{eq:u1InftyDefn-SchroEqu2018}
	\end{align}
	where
	\begin{equation} \label{eq:Fjkx-SchroEqu2018}
	\left\{\begin{aligned}
	F_0(k,\hat x,\omega) & := \int_D e^{-ik \hat x \cdot y} (\sigma \dot{B}_y) \dif{y}, \\
	F_1(k,\hat x,\omega) & := \sum_{j \geq 1} \int_D e^{-ik \hat x \cdot y} (V \Rk)^j (\sigma \dot{B}_y) \dif{y}.
	\end{aligned}\right.
	\end{equation}
	Meanwhile, the expectation of the far-field pattern $\mathbb E u^\infty$ is
	\begin{equation} \label{eq:u2InftyDefn-SchroEqu2018}
	\mathbb E u^\infty(\hat x,k) = \frac {-1} {4\pi} \int_D e^{-ik\hat x \cdot y} (I - V \Rk)^{-1} (f) \dif{y}, \quad \hat x \in \mathbb{S}^2.
	\end{equation}}
	
	\sq{
	\begin{lem} \label{lemma:FarFieldGoToZero-SchroEqu2018}
		We have
		\[
		\lim_{k \to +\infty} |\mathbb E u^\infty(\hat x,k)| = 0 \quad \text{uniformly in } \hat x \in \mathbb{S}^2.
		\]
	\end{lem}
	\begin{proof}[Proof of Lemma \ref{lemma:FarFieldGoToZero-SchroEqu2018}]
		Due to the fact that $f \in L^\infty(D) \subset L^2(D)$, we know
		\begin{equation} \label{eq:fApprox-SchroEqu2018}
		\Forall \epsilon > 0, \Exists \varphi_\epsilon \in \scrD(D), \st \nrm[L^2(D)]{f-\varphi_\epsilon} < \epsilon / (2 |D|^{\frac 1 2}).
		\end{equation}
		Recall that $k > k^*$, so $(I - V \Rk)^{-1}$ equals to $I + \sum_{j=1}^{+\infty} (V\Rk)^j$.
		By \eqref{eq:fApprox-SchroEqu2018} and Lemma \ref{lemma:RkVBounded-SchroEqu2018} and utilizing the stationary phase lemma, one can deduce as follows,
		\begin{align}
		|\mathbb E u^\infty(\hat x,k)|
		& \lesssim \big| \int_D e^{-ik \hat x \cdot y} \varphi_\epsilon(y) \dif{y} \big| + \big| \int_D e^{-ik \hat x \cdot y} \big[ f(y) - \varphi_\epsilon(y) + \big( \sum_{j \geq 1} (V\Rk)^j f \big) (y) \big] \dif{y} \big| \nonumber\\
		& \lesssim \big| k^{-2} \int_D e^{-ik \hat x \cdot y} \cdot \Delta \varphi_\epsilon(y) \dif{y} \big| + |D|^{\frac 1 2} \cdot \nrm[L^2(D)]{f - \varphi_\epsilon + \sum_{j \geq 1} (V\Rk)^j f } \nonumber\\
		& \leq k^{-2} \cdot |D|^{\frac 1 2} \cdot \nrm[L^2(D)]{\Delta \varphi_\epsilon} + |D|^{\frac 1 2} \cdot \big( \epsilon/(2 |D|^{\frac 1 2}) + C \sum_{j \geq 1} k^{-j} \nrm[L^2(D)]{f} \big) \nonumber\\
		& = k^{-2} \cdot |D|^{\frac 1 2} \nrm[L^2(D)]{\Delta \varphi_\epsilon} + \epsilon/2 + C (k-1)^{-1} \cdot \nrm[L^2(D)]{f}. \label{eq:u2InftyEst-SchroEqu2018}
		\end{align}
		Write $\mathcal K := \max\{ K_0, \frac 2 {\sqrt{\epsilon}} |D|^{\frac 1 4} \nrm[L^2(D)]{\Delta \varphi_\epsilon}^{\frac 1 2}, 1 + \frac {4C} \epsilon \nrm[L^2(D)]{f} \}$. From (\ref{eq:u2InftyEst-SchroEqu2018}) we have
		$$\Forall k > \mathcal K, \quad |\mathbb E u^\infty(\hat x,k)| < \frac \epsilon 2 + \frac \epsilon 4 + \frac \epsilon 4 = \epsilon, \quad \text{uniformly for } \forall \hat x \in \mathbb{S}^2.$$
		Since the $\epsilon$ is taken arbitrarily, the conclusion follows.
	\end{proof}}
	
	\sq{By substituting \eqref{eq:u1-SchroEqu2018}-\eqref{eq:u2InftyDefn-SchroEqu2018} into \eqref{eq:example-SchroEqu2018}, we obtain several crossover terms among $F_0$, $F_1$ and $\mathbb E u^\infty$. 
	The asymptotic estimates of these crossover terms are the main purpose of Sections \ref{subsec:AELeading-SchroEqu2018} and \ref{subsec:AEHigher-SchroEqu2018}.
	Section \ref{subsec:AELeading-SchroEqu2018} focuses on the estimate of the leading order term while the estimates of the higher order terms are presented in Section \ref{subsec:AEHigher-SchroEqu2018}.}

	\subsection{Asymptotic estimates of the leading order term} \label{subsec:AELeading-SchroEqu2018}
	\jz{Lemma \ref{lemma:LeadingTermErgo-SchroEqu2018} below is the asymptotic estimate of the crossover leading order term. By utilizing the ergodicity, the result of Lemma \ref{lemma:LeadingTermErgo-SchroEqu2018} is also statistically stable. To prove Lemma \ref{lemma:LeadingTermErgo-SchroEqu2018}, we need Lemmas \ref{lem:asconverg-SchroEqu2018}, \ref{lem:IsserlisThm-SchroEqu2018} and \ref{lemma:LeadingTermTechnical-SchroEqu2018}. Lemma \ref{lem:asconverg-SchroEqu2018} is the probabilistic foundation of our single-realization recovery result, and Lemma \ref{lem:IsserlisThm-SchroEqu2018} is called Isserlis' Theorem. In order to keep our arguments flowing, we postpone Lemma \ref{lemma:LeadingTermTechnical-SchroEqu2018} until we finish Lemma \ref{lemma:LeadingTermErgo-SchroEqu2018}.}
	
	\begin{lem} \label{lem:asconverg-SchroEqu2018}
		Assume $X$ and $X_n ~(n=1,2,\cdots)$ be complex-valued random variables, then
		$$X_n \to X \ass
		\quad\text{if and only if}\quad
		\lim_{K_0 \to +\infty} P \big( \bigcup_{j \geq K_0} \{ |X_j - X| \geq \epsilon \} \big) = 0 ~\Forall \epsilon > 0.$$
	\end{lem}
	\sq{The proof of Lemma \ref{lem:asconverg-SchroEqu2018} can be found in [\citen{dudley2002real}, Lemma 9.2.4].}
	
	\begin{lem}[Isserlis' Theorem \cite{Michalowicz2009}] \label{lem:IsserlisThm-SchroEqu2018}
		Suppose ${(X_{1},\dots, X_{2n})}$ is a zero-mean multi-variate normal random vector, then
		\[
		\mathbb{E} (X_1 X_2 \cdots X_{2n}) = \sum\prod \mathbb{E} (X_i X_j),
		\quad
		\mathbb{E} (X_1 X_2 \cdots X_{2n-1}) = 0.
		\]
		Specially,
		\[
		\mathbb{E} (\,X_{1}X_{2}X_{3}X_{4}\,)
		= \mathbb{E} (X_{1}X_{2})\, \mathbb{E} (X_{3}X_{4}) + \mathbb{E} (X_{1}X_{3})\, \mathbb{E} (X_{2}X_{4}) + \mathbb{E} (X_{1}X_{4})\, \mathbb{E} (X_{2}X_{3}).
		\]
	\end{lem}
	\jz{The proof of Lemma \ref{lem:IsserlisThm-SchroEqu2018} can be found in \cite{Michalowicz2009}.} In what follows, $\widehat{\varphi}$ denotes the Fourier transform of the function $\varphi$ defined as
	\begin{equation*} 
	\widehat{\varphi}(\xi) := (2\pi)^{-n/2} \int_{\R^3} e^{-i x \cdot \xi} \varphi(x) \dif{x}, \quad \xi \in \R^n.
	\end{equation*}

	\medskip

	For the notational convenience, we use ``$\{K_j\} \in P(t)$'' to mean that the sequence $\{K_j\}_{j \in \mathbb{N}^+}$ satisfies $K_j \geq C j^t ~(j \in \mathbb{N}^+)$ for some fixed constant $C > 0$. Throughout the following context, $\gamma$ stands for any fixed positive real number. Lemma \ref{lemma:LeadingTermErgo-SchroEqu2018} gives the asymptotic estimates of the crossover leading order term.
	\begin{lem} \label{lemma:LeadingTermErgo-SchroEqu2018}
		Write
		\begin{equation*}
		X_{0,0}(K,\tau,\hat x,\omega) = \frac 1 K \int_K^{2K} \overline{F_0(k,\hat x,\omega)} \cdot F_0(k+\tau,\hat x,\omega) \dif{k}.
		\end{equation*}
		Assume $\{K_j\} \in P(2+\gamma)$, then for any $\tau > 0$, we have
		\begin{equation*} 
		\lim_{j \to +\infty} X_{0,0}(K_j,\tau,\hat x,\omega) = (2\pi)^{3/2} \widehat{\sigma^2} (\tau \hat x) \quad \ass.
		\end{equation*}
	\end{lem}
	We may denote $X_{0,0}(K,\tau,\hat x,\omega)$ as $X_{0,0}$ for short if it is clear in the context.
	\begin{proof}[Proof of Lemma \ref{lemma:LeadingTermErgo-SchroEqu2018}]
		We have
		\begin{align}
		& \ \mathbb{E} \big( \overline{F_0(k,\hat x,\omega)} F_0(k+\tau,\hat x,\omega) \big) \nonumber\\
		= & \ \mathbb{E} \big( \int_{D_y} e^{ik \hat x \cdot y} \sigma(y) \dif{B_y} \cdot \int_{D_z} e^{-i(k+\tau) \hat x \cdot z} \sigma(z) \dif{B_z} \big) \nonumber\\
		= & \ \int_{D} e^{ik \hat x \cdot y} e^{-i(k+\tau) \hat x \cdot y} \sigma(y) \sigma(y) \dif{y} = (2\pi)^{3/2}\, \widehat{\sigma^2} (\tau \hat x). \label{eq:I0-SchroEqu2018}
		\end{align}
		
		From \eqref{eq:I0-SchroEqu2018} we conclude that
		\begin{align*}
		\mathbb{E} ( X_{0,0} )
		= \frac 1 K \int_K^{2K} \mathbb{E} \big( \overline{F_0(k,\hat x,\omega)} F_0(k+\tau,\hat x,\omega) \big) \dif{k}
		= (2\pi)^{3/2} \widehat{\sigma^2} (\tau \hat x).
		\end{align*}
		By Isserlis' Theorem and \eqref{eq:I0-SchroEqu2018}, and note that $\overline{F_j(k,\hat x,\omega)} = F_j(-k,\hat x,\omega)$, one can compute
		\begin{align}
		& \mathbb{E} \big( | X_{0,0} - (2\pi)^{3/2} \widehat{\sigma^2} (\tau \hat x) |^2 \big) \nonumber\\
		= & \frac 1 {K^2} \int_K^{2K} \int_K^{2K} \mathbb{E} \Big( \overline{F_0(k_1,\hat x,\omega)} F_0(k_1+\tau,\hat x,\omega) F_0(k_2,\hat x,\omega) \overline{F_0(k_2+\tau,\hat x,\omega)} \Big) \dif{k_1} \dif{k_2} \nonumber\\
		& - (2\pi)^3 |\widehat{\sigma^2} \big( \tau \hat x \big)|^2 - (2\pi)^3 |\widehat{\sigma^2} \big( \tau \hat x \big)|^2 + (2\pi)^3 |\widehat{\sigma^2} \big( \tau \hat x \big)|^2 \hspace{1.5cm} (\text{by } \eqref{eq:I0-SchroEqu2018}) \nonumber\\
		= & \frac 1 {K^2} \int_K^{2K} \int_K^{2K} \mathbb{E} \big( \overline{F_0(k_1,\hat x,\omega)} F_0(k_1+\tau,\hat x,\omega) \big) \cdot \mathbb{E} \big( F_0(k_2,\hat x,\omega) \overline{F_0(k_2+\tau,\hat x,\omega)} \big) \nonumber\\
		& + \mathbb{E} \big( \overline{F_0(k_1,\hat x,\omega)} F_0(k_2,\hat x,\omega) \big) \cdot \mathbb{E} \big( F_0(k_1+\tau,\hat x,\omega) \overline{F_0(k_2+\tau,\hat x,\omega)} \big) \nonumber\\
		& + \mathbb{E} \big( \overline{F_0(k_1,\hat x,\omega)} F_0(-k_2-\tau,\hat x,\omega) \big) \cdot \mathbb{E} \big( \overline{F_0(-k_1-\tau,\hat x,\omega)} F_0(k_2,\hat x,\omega) \big) \dif{k_1} \dif{k_2} \nonumber\\
		& - (2\pi)^3 |\widehat{\sigma^2} \big( \tau \hat x \big)|^2 \nonumber\\
		= & \frac {(2\pi)^3} {K^2} \int\limits_K^{2K} \int\limits_K^{2K} |\widehat{\sigma^2}((k_2 - k_1) \hat x)|^2 \dif{k_1} \dif{k_2} + \frac {(2\pi)^3} {K^2} \int\limits_K^{2K} \int\limits_K^{2K} |\widehat{\sigma^2}((k_1 + k_2 + \tau) \hat x)|^2 \dif{k_1} \dif{k_2}. \label{eq:X00Square-SchroEqu2018}
		\end{align}
		\jz{ Note that $|\widehat{\sigma^2} \big( (k_1 - k_2) \hat x \big)| = |\widehat{\sigma^2} \big( -(k_1 - k_2) \hat x \big)|$.} Combining (\ref{eq:X00Square-SchroEqu2018}) and Lemma \ref{lemma:LeadingTermTechnical-SchroEqu2018}, we have
		\begin{equation} \label{eq:X00Bdd-SchroEqu2018}
		\mathbb{E} \big( | X_{0,0} - (2\pi)^{3/2} \widehat{\sigma^2} (\tau \hat x) |^2 \big) = \mathcal{O}(K^{-1/2}), \quad K \to +\infty.
		\end{equation}
		For any integer $K_0 > 0$, by Chebyshev's inequality and (\ref{eq:X00Bdd-SchroEqu2018}) we have
		\begin{align}
		& P \big( \bigcup_{j \geq K_0} \{ | X_{0,0}(K_j) - (2\pi)^{3/2} \widehat{\sigma^2} (\tau \hat x) | \geq \epsilon \} \big) \leq \frac 1 {\epsilon^2} \sum_{j \geq K_0} \mathbb{E} \big( | X_{0,0}(K_j) - (2\pi)^{3/2} \widehat{\sigma^2} (\tau \hat x) |^2 \big) \nonumber\\
		\lesssim & \frac 1 {\epsilon^2} \sum_{j \geq K_0} K_j^{-1/2} = \frac 1 {\epsilon^2} \sum_{j \geq K_0} j^{-1-\gamma/2} \leq \frac 1 {\epsilon^2} \int_{K_0}^{+\infty} (t-1)^{-1-\gamma/2} \dif{t} = \frac 2 {\epsilon^2 \gamma} (K_0-1)^{-\gamma/2}. \label{eq:PX00Epsilon-SchroEqu2018}
		\end{align}
		Here $X_{0,0}(K_j)$ stands for $X_{0,0}(K_j, \tau, \hat x,\omega)$. By Lemma \ref{lem:asconverg-SchroEqu2018}, formula (\ref{eq:PX00Epsilon-SchroEqu2018}) implies that for any fixed $\tau \geq 0$ and fixed $\hat x \in \mathbb{S}^2$, we have
		$$X_{0,0}(K_j,\tau,\hat x,\omega) \to (2\pi)^{3/2} \widehat{\sigma^2} (\tau \hat x)\quad \ass.$$
		The proof is done.
	\end{proof}

	\medskip

	Lemma \ref{lemma:LeadingTermTechnical-SchroEqu2018} plays a critical role in the estimates of the leading order term.
	\begin{lem} \label{lemma:LeadingTermTechnical-SchroEqu2018}
		\jz{Assume that} $\tau \geq 0$ is fixed, then $\exists K_0 > \tau$, and $K_0$ is independent of $\hat x$, such that for all $K > K_0$, we have the following estimates:
		\begin{align}
		\frac {(2\pi)^3} {K^2} \int_K^{2K} \int_K^{2K} \big| \widehat{\sigma^2}((k_1 - k_2) \hat x) \big|^2 \dif{k_1} \dif{k_2} & \leq CK^{-1/2}, \label{eq:F0F0TermOne-SchroEqu2018} \\
		\frac {(2\pi)^3} {K^2} \int_K^{2K} \int_K^{2K} \big| \widehat{\sigma^2}((k_1 + k_2 + \tau) \hat x) \big|^2 \dif{k_1} \dif{k_2} & \leq CK^{-1/2}, \label{eq:F0F0TermTwo-SchroEqu2018}
		\end{align}
		for some constant $C$ independent of $\tau$ and $\hat x$.
	\end{lem}
	\begin{proof}[Proof of Lemma \ref{lemma:LeadingTermTechnical-SchroEqu2018}]
		Note that for every $x \in \R^3$, we have
		\begin{equation*} 
		|\widehat{\sigma^2}(x)|^2 \simeq \big| \int_{\R^3} e^{-i x \cdot \xi} \sigma^2(\xi) \dif{\xi} \big|^2
		\leq \big( \int_{\R^3} |\sigma^2(\xi)| \dif{\xi} \big)^2
		\leq \nrm[L^\infty(D)]{\sigma}^4 \cdot |D|^2.
		\end{equation*}
		
		To conclude (\ref{eq:F0F0TermOne-SchroEqu2018}), we make a change of variable,
		\begin{equation*}
		\left\{\begin{aligned}
		s & = k_1 - k_2, \\
		t & = k_2.
		\end{aligned}\right.
		\end{equation*}
		Write \jz{$Q = \{(s,t) \in \R^2 \,\big|\, K \leq s+t \leq 2K,\, K \leq t \leq 2K \}$}. $Q$ is illustrated as in Figure \ref{fig:D-SchroEqu2018}.
		\begin{figure}[h]
			\centering
			\begin{tikzpicture}
			\draw[dashed,->] (-4.5,-1) -- (-3.5,-1) node[anchor=north west] {$s$};
			\draw[dashed,->] (-5,-0.5) -- (-5,0.5) node[anchor=south west] {$t$};
			\filldraw[opacity=0.5, thick] (0,-1) -- (2,-1) -- (0,1) -- (-2,1) -- (0,-1);
			\filldraw[black] (0,-1) node[anchor=north east] {$(0,K)$};
			\filldraw[black] (2,-1) node[anchor=north west] {$(K,K)$};
			\filldraw[black] (0, 1) node[anchor=south west] {$(0,2K)$};
			\filldraw[black] (-2,1) node[anchor=south east] {$(-K,2K)$};
			\end{tikzpicture}
			\caption{Illustration of $Q$} \label{fig:D-SchroEqu2018}
		\end{figure}
		
		Recall that $\supp \sigma \subseteq D$, so we have
		\begin{align}
		& \frac 1 {K^2} \int_K^{2K} \int_K^{2K} |\widehat{\sigma^2}((k_1 - k_2) \hat x)|^2 \dif{k_1} \dif{k_2} 
		= \frac 1 {K^2} \iint_Q \big| \widehat{\sigma^2}(s \hat x) \big|^2 \dif{s} \dif{t} 
		\nonumber\\
		= \ & \frac 1 {K^2} \int_{-K}^0 (K+s) |\widehat{\sigma^2}(s \hat x)|^2 \dif{s} + \frac 1 {K^2} \int_0^{K} (K-s) |\widehat{\sigma^2}(s \hat x)|^2 \dif{s} \nonumber\\
		\simeq \ & \int_0^1 \Big( \int_D e^{-iKs \hat x \cdot y} \sigma^2(y) \dif{y} \cdot \int_D e^{iKs \hat x \cdot z} \sigma^2(z) \dif{z} \Big) \dif{s} \nonumber\\
		= \ & \int_{(D \times D) \backslash E_\epsilon} \Big( \int_0^1 e^{iK(\hat x \cdot z - \hat x \cdot y)s} \dif{s} \Big) \sigma^2(y) \sigma^2(z) \dif{y} \dif{z} \nonumber\\
		& \quad + \int_{E_\epsilon} \Big( \int_0^1 e^{iK(\hat x \cdot z - \hat x \cdot y)s} \dif{s} \Big) \sigma^2(y) \sigma^2(z) \dif{y} \dif{z} \nonumber\\
		=: & A_1 + A_2, \label{eq:sigma2Inter-SchroEqu2018}
		\end{align}
		where $E_\epsilon := \{ (y,z) \in D \times D ; |\hat x \cdot z - \hat x \cdot y| < \epsilon \}$. We first estimate $A_1$,
		\begin{align}
		|A_1|
		& = \Big| \int_{(D \times D) \backslash E_\epsilon} \Big( \int_0^1 e^{iK(\hat x \cdot z - \hat x \cdot y)s} \dif{s} \Big) \sigma^2(y) \sigma^2(z) \dif{y} \dif{z} \Big| \nonumber\\
		& \leq \int_{(D \times D) \backslash E_\epsilon} \Big| \frac {e^{iK(\hat x \cdot z - \hat x \cdot y)} - 1} {iK(\hat x \cdot z - \hat x \cdot y)} \sigma^2(y) \sigma^2(z) \Big| \dif{y} \dif{z} \nonumber\\
		& \leq \frac 2 {K\epsilon} \nrm[L^\infty(D)]{\sigma}^4 \int_{D \times D} 1 \dif{y} \dif{z} = \frac {2|D|^2} {K\epsilon} \nrm[L^\infty(D)]{\sigma}^4. \label{eq:sigma2InterA1-SchroEqu2018}
		\end{align}
		\jz{Recall that $\text{diam}\,D < +\infty$ and that the problem setting is in $\R^3$. We can estimate $A_2$ as}
		\begin{align}
		|A_2|
		& \leq \nrm[L^\infty(D)]{\sigma}^4 \int_{E_\epsilon} 1 \dif{y} \dif{z} \nonumber\\
		& = \nrm[L^\infty(D)]{\sigma}^4 \int_D \big( \int_{y \in D \,,\, |\hat x \cdot z - \hat x \cdot y| < \epsilon} 1 \dif{y} \big) \dif{z} 
		\nonumber\\
		& \leq \nrm[L^\infty(D)]{\sigma}^4 \int_D 2\epsilon (\text{Diam}\,D)^2 \dif{z} \nonumber\\
		& \leq 2\nrm[L^\infty(D)]{\sigma}^4 (\text{Diam}\,D)^2 |D| \cdot \epsilon. \label{eq:sigma2InterA2-SchroEqu2018}
		\end{align}
		Set $\epsilon = K^{-1/2}$. By \eqref{eq:sigma2Inter-SchroEqu2018}-\eqref{eq:sigma2InterA2-SchroEqu2018}, we arrive at
		$$\frac 1 {K^2} \int\limits_K^{2K} \int\limits_K^{2K} |\widehat{\sigma^2}((k_1 - k_2) \hat x)|^2 \dif{k_1} \dif{k_2} \leq C K^{-1/2},$$
		for some constant $C$ independent of $\hat x$.

		\smallskip

		Now we prove \eqref{eq:F0F0TermTwo-SchroEqu2018}. Similarly, we make a change of variable:
		\begin{equation*} 
		\left\{\begin{aligned}
		s & = k_1 + k_2 + \tau, \\
		t & = k_2.
		\end{aligned}\right.
		\end{equation*}
		Write \jz{$Q' = \{(s,t) \in \R^2 \,\big|\, K \leq s-t-\tau \leq 2K,\, K \leq t \leq 2K \}$}. One can compute
		\begin{align*}
		& \frac 1 {K^2} \int_K^{2K} \int_K^{2K} | \widehat{\sigma^2}((k_1 + k_2 + \tau) \hat x) |^2 \dif{k_1} \dif{k_2} = \frac 1 {K^2} \iint_{Q'} | \widehat{\sigma^2}(s \hat x) |^2 \dif{s} \dif{s} \\
		= & \frac 1 {K^2} \int_{2K+\tau}^{3K+\tau} (s-2K-\tau) | \widehat{\sigma^2}(s \hat x) |^2 \dif{s} + \frac 1 {K^2} \int_{3K+\tau}^{4K+\tau} (4K+\tau-s) | \widehat{\sigma^2}(s \hat x) |^2 \dif{s} \\
		\leq & \frac 2 {K} \int_{2K-\tau}^{2K+\tau} | \widehat{\sigma^2}(s \hat x) |^2 \dif{s} = 2 \int_{2+\tau/K}^{4+\tau/K} | \widehat{\sigma^2}(Ks \hat x) |^2 \dif{s}.
		\end{align*}
		Thus when $K > \tau$,
		\begin{equation} \label{eq:sigma2InterTau-SchroEqu2018}
		\frac 1 {K^2} \int_K^{2K} \int_K^{2K} | \widehat{\sigma^2}((k_1 + k_2 + \tau) \hat x) |^2 \dif{k_1} \dif{k_2} \leq 2 \int_2^5 | \widehat{\sigma^2}(Ks \hat x) |^2 \dif{s}.
		\end{equation}
		Following the same manner as in \eqref{eq:sigma2Inter-SchroEqu2018}-\eqref{eq:sigma2InterA2-SchroEqu2018}, from \eqref{eq:sigma2InterTau-SchroEqu2018} we arrive at \eqref{eq:F0F0TermTwo-SchroEqu2018}. The proof is done.
	\end{proof}

	\subsection{Asymptotic estimates of higher order terms} \label{subsec:AEHigher-SchroEqu2018}
	The asymptotic estimates of the higher order terms are presented in Lemma \ref{lemma:HOT-SchroEqu2018}.
	
	\begin{lem} \label{lemma:HOT-SchroEqu2018}
		For every $\hat x_1$, $\hat x_2 \in \mathbb{S}^2$ and every $k_1$, $k_2 \geq k$, we have the following estimates ($j = 0,1$) as $k \to +\infty$,
		\begin{align}
		\big| \mathbb{E} \big( \overline{F_j(k_1,\hat x_1,\omega)} \cdot F_1(k_2,\hat x_2,\omega) \big) \big| & = \mathcal{O}(k^{-1}), \label{eq:hotFjF1-SchroEqu2018}\\
		\big| \mathbb{E} \big( F_j(k_1,\hat x_1,\omega) \cdot F_1(k_2,\hat x_2,\omega) \big) \big| & = \mathcal{O}(k^{-1}). \label{eq:hotFjF1Conju-SchroEqu2018}
		\end{align}
	\end{lem}
	\begin{proof}[Proof of Lemma \ref{lemma:HOT-SchroEqu2018}]
		The proof of formulas \eqref{eq:hotFjF1Conju-SchroEqu2018} is similar to that of \eqref{eq:hotFjF1-SchroEqu2018}, so we only present the proof of \eqref{eq:hotFjF1-SchroEqu2018}. In this proof, we may drop the arguments $k$, $\hat x$ or $\omega$ from $F_j$ if it is clear in the context. For the notational convenience, we write
		\begin{align*}
		G_j(k,\hat x,\omega) & := \int_D e^{-ik \hat x \cdot y} (V \Rk)^j (\sigma \dot{B}_y) \dif{y},
		\\
		r_j(k,\hat x,\omega) & := \sum_{s \geq j} G_s(k,\hat x,\omega), 
		\end{align*}
		for $j = 0,1,\cdots$. To prove \eqref{eq:hotFjF1-SchroEqu2018} for the case where $j = 0$, we first show that
		\begin{equation} \label{eq:hotF0Fj-SchroEqu2018}
		\mathbb{E} \big( \overline{G_0(k_1,\hat x_1,\omega)} \cdot G_j(k_2,\hat x_2,\omega) \big) = \int_D e^{-ik_2 \hat x_2 \cdot z} (V \mathcal{R}_{k_2})^j \big( e^{ik_1 \hat x_1 \cdot (\cdot)} \sigma^2 \big) \dif{z}, \quad j \geq 1.
		\end{equation}
		This can be seen from the following computation
		\begin{align}
		& \ \mathbb{E} \big( \overline{G_0(k_1,\hat x_1,\omega)} \cdot G_j(k_2,\hat x_2,\omega) \big) \nonumber\\
		= & \ \mathbb{E} \big( \int_D e^{ik_1 \hat x_1 \cdot y} \sigma(y) \dif{B_y} \cdot \int_D \big[ e^{-ik_2 \hat x_2 \cdot z} (V \mathcal{R}_{k_2})^{j-1} ( V(\cdot) \int_{D_s} \Phi(\cdot,s) \sigma(s) \dif{B_s} ) \big] \dif{z} \big) \nonumber\\
		= & \int_D e^{-ik_2 \hat x_2 \cdot z} (V \mathcal{R}_{k_2})^{j-1} \Big\{ V(\cdot) \,\mathbb{E} \big[ \int_{D_y} e^{ik_1 \hat x_1 \cdot y} \sigma(y) \dif{B_y} \cdot \int_{D_s} \Phi(\cdot,s) \sigma(s) \dif{B_s} \big] \Big\} \dif{z} \nonumber\\
		= & \int_D e^{-ik_2 \hat x_2 \cdot z} (V \mathcal{R}_{k_2})^{j-1} \big( V(\cdot) \mathcal{R}_{k_2}(e^{ik_1 \hat x_1 \cdot (\cdot)} \sigma^2) \big) \dif{z} \nonumber\\
		= & \int_D e^{-ik_2 \hat x_2 \cdot z} (V \mathcal{R}_{k_2})^j ( e^{ik_1 \hat x_1 \cdot (\cdot)} \sigma^2 ) \dif{z}. \label{eq:hotF0FjInter-SchroEqu2018}
		\end{align}
		From \eqref{eq:hotF0FjInter-SchroEqu2018}, equality \eqref{eq:hotF0Fj-SchroEqu2018} is proved. Using \eqref{eq:hotF0Fj-SchroEqu2018} and Lemma \ref{lemma:RkVBounded-SchroEqu2018}, we have
		\begin{align*}
		& \ \big| \mathbb{E} \big( \overline{F_0(k_1,\hat x_1,\omega)} \cdot F_1(k_2,\hat x_2,\omega) \big) \big| \nonumber\\
		\leq & \ \sum_{j \geq 1} \big| \mathbb{E} \big( G_0(k_1,\hat x_1,\omega) \cdot \overline{G_j(k_2,\hat x_2,\omega)} \big) \big| \nonumber\\
		= & \ \sum_{j \geq 1} \Big| \int_D e^{-ik_2 \hat x_2 \cdot z} (V \mathcal{R}_{k_2})^j \big( e^{ik_1 \hat x_1 \cdot (\cdot)} \sigma^2 \big) \dif{z} \Big| \nonumber\\
		\leq & \ |D|^{1/2} \cdot \sum_{j \geq 1} \nrm[L^2(D)]{ (V \mathcal{R}_{k_2})^j \big( e^{ik_1 \hat x_1 \cdot (\cdot)} \sigma^2 \big) } \nonumber\\
		\leq & \ C |D|^{1/2} \cdot \sum_{j \geq 1} k_2^{-j} \nrm[L^2(D)]{e^{ik_1 \hat x_1 \cdot (\cdot)} \sigma^2} = \mathcal{O}(k_2^{-1}), \quad k \to +\infty. 
		\end{align*}

		\medskip

		\sq{To prove \eqref{eq:hotFjF1-SchroEqu2018} for the case where $j = 1$, we split $\mathbb{E} (\overline{F_1} F_1)$ into four terms,
		\begin{equation} \label{eq:FGr-SchroEqu2018}
		\mathbb{E} (\overline{F_1} F_1) = \mathbb{E} (\overline{G_1} G_1) + \mathbb{E} (\overline{r_1} r_2) - \mathbb{E} (\overline{r_2} r_2) + \mathbb{E} (\overline{r_2} r_1).
		\end{equation}
		We estimate these four terms on the right-hand-side of \eqref{eq:FGr-SchroEqu2018} one by one.}
		First, we estimate
		\begin{align}
		& \ \big| \mathbb{E} \big( \overline{G_1(k_1,\hat x_1,\omega)} \cdot G_1(k_2,\hat x_2,\omega) \big) \big| \nonumber\\
		= & \ \Big| \iint_{D_y \times D_z} e^{-ik_1 \hat x_1 \cdot y} e^{ik_2 \hat x_2 \cdot z} V(y) \overline V(z) \cdot \mathbb{E} \big[ \int_{D_s} \Phi(y,s) \sigma(s) \dif{B_s} \cdot \int_{D_t} \overline \Phi(z,t) \sigma(t) \dif{B_t} \big] \dif{y} \dif{z} \Big| \nonumber\\
		= & \ \Big| \iint_{D_y \times D_z} e^{-ik_1 \hat x_1 \cdot y} e^{ik_2 \hat x_2 \cdot z} V(y) \overline V(z) \cdot \big[ \int_{D_s} \Phi(y,s) \sigma(s) \overline \Phi(z,s) \sigma(s) \dif{s} \big] \dif{y} \dif{z} \Big| \nonumber\\
		= & \ \Big| \int_{D} \sigma^2(s) \cdot \mathcal{R}_{k_1} V( e^{-ik_1 \hat x_1 \cdot (\cdot)} )(s) \cdot \overline{ \mathcal{R}_{k_2} V ( e^{-ik_2 \hat x_2 \cdot (\cdot)} )(s) } \dif{s} \Big| \nonumber\\
		\leq & \ C k_1^{-1} k_2^{-1} \nrm[L^\infty(D)]{\sigma}^2 \quad\big( \text{Lemma \ref{lemma:RkVBounded-SchroEqu2018}} \big) \nonumber\\
		= & \ \mathcal{O}(k_1^{-1} k_2^{-1}), \quad k \to +\infty. \label{eq:hotG1G1-SchroEqu2018}
		\end{align}
		Then we estimate
		\begin{align}
		& \ \big| \mathbb{E} \big( \overline{r_1(k_1,\hat x_1,\omega)} \cdot r_2(k_2,\hat x_2,\omega) \big) \big|
		\leq \mathbb{E} \Big( \sum_{j \geq 1} \big| G_j(k_1,\hat x_1,\omega) \big| \times \sum_{\ell \geq 2} \big| G_\ell(k_2,\hat x_2,\omega) \big| \Big) \nonumber\\
		= & \ \mathbb{E} \Big( \sum_{j \geq 1} \big| \int_D e^{-ik_1 \hat x_1 \cdot y} (V \mathcal{R}_{k_1})^j (\sigma \dot{B}_y) \dif{y} \big| \times \sum_{\ell \geq 2} \big| \int_D e^{-ik_2 \hat x_2 \cdot z} (V \mathcal{R}_{k_2})^\ell (\sigma \dot{B}_z) \dif{z} \big| \Big) \nonumber\\
		= & \ \nrm[L^\infty(D)]{V}^2 |D| \cdot \mathbb{E} \Big( \sum_{j \geq 0} \nrm[L^2(D)]{(\mathcal{R}_{k_1} V)^j [\mathcal{R}_{k_1}(\sigma \dot{B})]} \times \sum_{\ell \geq 1} \nrm[L^2(D)]{(\mathcal{R}_{k_2} V)^\ell [\mathcal{R}_{k_2}(\sigma \dot{B})]} \Big) \nonumber\\
		\leq & \ C \nrm[L^\infty(D)]{V}^2 |D| \cdot \mathbb{E} \Big( \sum_{j \geq 0} \big( k_1^{-j} \nrm[L^2(D)]{\mathcal{R}_{k_1}(\sigma \dot{B})} \big) \times \sum_{\ell \geq 1} \big( k_2^{-\ell} \nrm[L^2(D)]{\mathcal{R}_{k_2}(\sigma \dot{B})} \big) \Big) \nonumber\\
		\leq & \ \nrm[L^\infty(D)]{V}^2 |D| \cdot \frac {k_1} {k_1-1} \cdot \frac 1 {k_2-1} \cdot \frac 1 2 \mathbb{E} \big( \nrm[L^2(D)]{\mathcal{R}_{k_1}(\sigma \dot{B})}^2 + \nrm[L^2(D)]{\mathcal{R}_{k_2}(\sigma \dot{B})}^2 \big). \label{eq:hotr1r2Inter1-SchroEqu2018}
		\end{align}
		Utilizing \eqref{eq:Rksigma2Bounded-SchroEqu2018}, we obtain
		\begin{equation}
		\mathbb{E} \big( \nrm[L^2(D)] {\Rk(\sigma \dot{B})}^2 \big)
		\leq C \mathbb{E} \big( \nrm[L_{-1/2-\epsilon}^2] {\Rk(\sigma \dot{B})}^2 \big) \leq C_D < +\infty. \label{eq:hotr1r2Inter2-SchroEqu2018}
		\end{equation}
		From (\ref{eq:hotr1r2Inter1-SchroEqu2018})-(\ref{eq:hotr1r2Inter2-SchroEqu2018}) we arrive at
		\begin{equation} \label{eq:hotr1r2-SchroEqu2018}
		\big| \mathbb{E} \big( \overline{r_1(k_1,\hat x_1,\omega)} \cdot r_2(k_2,\hat x_2,\omega)\big) \big| \leq \mathcal{O}(k_2^{-1}), \quad k \to +\infty.
		\end{equation}
		Mimicking (\ref{eq:hotr1r2Inter1-SchroEqu2018})-(\ref{eq:hotr1r2Inter2-SchroEqu2018}), one can obtain
		\begin{equation} \label{eq:hotr2r1-SchroEqu2018}
		\big| \mathbb{E} \big( \overline{r_2(k_1,\hat x_1,\omega)} \cdot r_1(k_2,\hat x_2,\omega) \big) \big| \leq \mathcal{O}(k_1^{-1}), \quad k \to +\infty.
		\end{equation}
		By modify $\sum_{j \geq 0} k_1^{-j}$ to $\sum_{j \geq 1} k_1^{-j}$ in (\ref{eq:hotr1r2Inter1-SchroEqu2018}), one can conclude
		\begin{equation} \label{eq:hotr2r2-SchroEqu2018}
		\big| \mathbb{E} \big( \overline{r_2(k_1,\hat x_1,\omega)} \cdot r_2(k_2,\hat x_2,\omega) \big) \big| \leq \mathcal{O}(k_1^{-1}k_2^{-1}), \quad k \to +\infty.
		\end{equation}
		Combining \eqref{eq:FGr-SchroEqu2018}-\eqref{eq:hotG1G1-SchroEqu2018} and \eqref{eq:hotr1r2-SchroEqu2018}-\eqref{eq:hotr2r2-SchroEqu2018}, we arrive at \eqref{eq:hotFjF1-SchroEqu2018} for the case where $j = 1$. The proof is complete.
	\end{proof}

	\medskip

	Lemma \ref{lemma:HOTErgo-SchroEqu2018} is the ergodic version of Lemma \ref{lemma:HOT-SchroEqu2018}.
	\begin{lem} \label{lemma:HOTErgo-SchroEqu2018}
		Write
		\begin{align*}
		X_{p,q}(K,\tau,\hat x,\omega) & = \frac 1 K \int_K^{2K} \overline{F_p(k,\hat x,\omega)} \cdot F_q(k+\tau,\hat x,\omega) \dif{k}, \ \ \text{for} \ \ (p,q) \in \{ (0,1), (1,0), (1,1) \}.
		\end{align*}
		Then for any $\hat x \in \mathbb{S}^2$ and any $\tau \geq 0$, we have the following estimates as $K \to +\infty$,
		\begin{align} 
		\big| \mathbb{E} (X_{p,q}(K,\tau,\hat x,\omega)) \big| & = \mathcal{O}(K^{-1}), \ \mathbb{E} (|X_{p,q}(K,\tau,\hat x,\omega)|^2) = \mathcal{O}(K^{-5/4}), 
		\label{eq:hotF0F1Ergo-SchroEqu2018} \\
		\big| \mathbb{E} (X_{1,1}(K,\tau,\hat x,\omega)) \big| & = \mathcal{O}(K^{-1}), \ \mathbb{E} (|X_{1,1}(K,\tau,\hat x,\omega)|^2) = \mathcal{O}(K^{-2}), \label{eq:hotF1F1Ergo-SchroEqu2018}
		\end{align}
		for $(p,q) \in \{ (0,1), (1,0) \}$. Let $\{K_j\} \in P(4/5+\gamma)$. Then for any $\tau \geq 0$, we have
		\begin{equation} \label{eq:HOTErgoToZero-SchroEqu2018}
		\lim_{j \to +\infty} X_{p,q}(K_j,\tau,\hat x,\omega) = 0 \quad \ass,
		\end{equation}
		for every $(p,q) \in \{ (0,1), (1,0), (1,1) \}$.
	\end{lem}
	
	We may denote $X_{p,q}(K,\tau,\hat x,\omega)$ as $X_{p,q}$ for short if it is clear in the context.
	
	\begin{proof}[Proof of Lemma \ref{lemma:HOTErgo-SchroEqu2018}] According to Lemma \ref{lemma:HOT-SchroEqu2018}, we have
		\begin{align} 
		\mathbb{E} \big( X_{0,1} \big) 
		& = \frac 1 K \int_K^{2K} \mathbb{E} \big( \overline{F_0(k,\hat x,\omega)} \cdot F_1(k+\tau,\hat x,\omega) \big) \dif{k} \nonumber\\
		& = \mathcal{O}(K^{-1}), \quad K \to +\infty.  \label{eq:hotF0F1Ergo1-SchroEqu2018}
		\end{align}
		By (\ref{eq:I0-SchroEqu2018}), Isserlis' Theorem and Lemma \ref{lemma:LeadingTermTechnical-SchroEqu2018}, we compute the secondary moment of $X_{0,1}$ as
		\begin{align}
		& \ \mathbb{E} \big( | X_{0,1} |^2 \big) \nonumber\\
		= & \ \frac 1 {K^2} \int_K^{2K} \int_K^{2K} \mathbb{E} \big( F_0(k_1,\hat x,\omega) \overline{F_1(k_1+\tau,\hat x,\omega)} \big) \cdot \mathbb{E} \big( \overline{F_0(k_2,\hat x,\omega)} F_1(k_2+\tau,\hat x,\omega) \big) \nonumber\\
		& \ + \mathbb{E} \big( F_0(k_1,\hat x,\omega) \overline{F_0(k_2,\hat x,\omega)} \big) \cdot \mathbb{E} \big( \overline{F_1(k_1+\tau,\hat x,\omega)} F_1(k_2+\tau,\hat x,\omega) \big) \nonumber\\
		& \ + \mathbb{E} \big( F_0(k_1,\hat x,\omega) F_1(k_2+\tau,\hat x,\omega) \big) \cdot \mathbb{E} \big( \overline{F_1(k_1+\tau,\hat x,\omega)} \, \overline{F_0(k_2,\hat x,\omega)} \big) \dif{k_1} \dif{k_2} \nonumber\\
		= & \ \frac 1 {K^2} \int_K^{2K} \int_K^{2K} \mathcal{O}(K^{-2}) + (2\pi)^{3/2} \widehat{\sigma^2} ((k_1-k_2) \hat x) \cdot \mathcal{O}(K^{-1}) + \mathcal{O}(K^{-2}) \dif{k_1} \dif{k_2} \nonumber\\
		= & \ \mathcal{O}(K^{-1/4}) \cdot \mathcal{O}(K^{-1}) + \mathcal{O}(K^{-2}) \quad(\text{H\"older ineq. and Lemma } \ref{lemma:LeadingTermTechnical-SchroEqu2018}) \nonumber\\
		= & \ \mathcal{O}(K^{-5/4}), \quad K \to +\infty. \label{eq:hotF0F1Ergo2-SchroEqu2018}
		\end{align}
		From \eqref{eq:hotF0F1Ergo1-SchroEqu2018}-\eqref{eq:hotF0F1Ergo2-SchroEqu2018} we obtain \eqref{eq:hotF0F1Ergo-SchroEqu2018} for the case where $(p,q) = (0,1)$. 
		Using similar arguments, formula \eqref{eq:hotF0F1Ergo-SchroEqu2018} for $(p,q) = (1,0)$ can be proved and we skip the details.
		
		By Chebyshev's inequality and (\ref{eq:hotF0F1Ergo2-SchroEqu2018}), for any $\epsilon > 0$, we have
		\begin{align}
		\qquad & \ P \big( \bigcup_{j \geq K_0} \{ |X_{0,1}(K_j, \tau, \hat x,\omega) - 0| \geq \epsilon \} \big) \leq \frac C {\epsilon^2} \sum_{j \geq K_0} K_j^{-5/4} \leq \frac C {\epsilon^2} \sum_{j \geq K_0} j^{-1-5\gamma/4} \nonumber\\
		\leq & \ \frac C {\epsilon^2} \int_{K_0}^{+\infty} (t-1)^{-1-5\gamma/4} \dif{t} = \frac C {\epsilon^2 \gamma} (K_0-1)^{-5\gamma/4} \to 0, \quad K_0 \to +\infty. \label{eq:X01Ergo-SchroEqu2018}
		\end{align}
		According to Lemma \ref{lem:asconverg-SchroEqu2018}, inequality \eqref{eq:X01Ergo-SchroEqu2018} implies \eqref{eq:HOTErgoToZero-SchroEqu2018} for the case where $(p,q) = (0,1)$. Similarly, formula \eqref{eq:HOTErgoToZero-SchroEqu2018} can be proved for the case where $(p,q) = (1,0)$.

		\medskip

		We now prove \eqref{eq:hotF1F1Ergo-SchroEqu2018}. We have
		\begin{align}
		\mathbb{E} \big( X_{1,1} \big)
		& = \frac 1 K \int_K^{2K} \mathbb{E} \big( \overline{F_1(k,\hat x,\omega)} \cdot F_1(k+\tau,\hat x,\omega) \big) \dif{k}	
		= \mathcal{O}(K^{-1}). \label{eq:hotF1F1Ergo1-SchroEqu2018}
		\end{align}
		Similar to \eqref{eq:hotF0F1Ergo2-SchroEqu2018}, we compute the secondary moment of $X_{1,1}$ as
		\begin{align}
		& \ \mathbb{E} \big( | X_{1,1} |^2 \big) \nonumber\\
		= & \ \mathbb{E} \big( \frac 1 K \int_K^{2K} F_1(k_1,\hat x,\omega) \cdot \overline{F_1(k_1+\tau,\hat x,\omega)} \dif{k_1} \cdot \frac 1 K \int_K^{2K} \overline{ F_1(k_2,\hat x,\omega) } \cdot F_1(k_2+\tau,\hat x,\omega) \dif{k_2} \big) \nonumber\\
		= & \ \frac 1 {K^2} \int_K^{2K} \int_K^{2K} \mathcal{O}(K^{-1}) \cdot \mathcal{O}(K^{-1}) \dif{k_1} \dif{k_2} \quad (\text{Lemma } \ref{lemma:HOT-SchroEqu2018}) \nonumber\\
		= & \ \mathcal{O}(K^{-2}), \quad K \to +\infty. \label{eq:hotF1F1Ergo2-SchroEqu2018}
		\end{align}
		Formulae \eqref{eq:hotF1F1Ergo1-SchroEqu2018} and \eqref{eq:hotF1F1Ergo2-SchroEqu2018} give \eqref{eq:hotF1F1Ergo-SchroEqu2018}.
		
		By Chebyshev's inequality and \eqref{eq:hotF1F1Ergo2-SchroEqu2018}, for any $\epsilon > 0$, we have
		\begin{align}
		& P \big( \bigcup_{j \geq K_0} \{ |X_{1,1} - 0| \geq \epsilon \} \big) \leq \frac C {\epsilon^2} \sum_{j \geq K_0} K_j^{-2} \leq \frac C {\epsilon^2} \sum_{j \geq K_0} j^{-8/5-2\gamma} \nonumber\\
		\leq & \frac C {\epsilon^2} \int_{K_0}^{+\infty} (t-1)^{-8/5-2\gamma} \dif{t} = \frac {C (K_0-1)^{-3/5-2\gamma}} {\epsilon^2 (3+10\gamma)}  \to 0, \quad K_0 \to +\infty.  \label{eq:X11Ergo-SchroEqu2018}
		\end{align}
		Lemma \ref{lem:asconverg-SchroEqu2018} together with \eqref{eq:X11Ergo-SchroEqu2018} implies \eqref{eq:HOTErgoToZero-SchroEqu2018} for the case that $(p,q) = (1,1)$. The proof is thus completed.
	\end{proof}

	\section{The recovery of the variance function} \label{sec:RecVar-SchroEqu2018}
	
	In this section we focus on the recovery of the variance function. We employ only a single passive scattering measurement. 
	Namely, there is no incident plane wave sent and the random sample $\omega$ is fixed. Throughout this section, $\alpha$ is set to be 0. 
	The data set $\mathcal M_1$ is utilized to achieve the unique recovery result. 
	We present the main results of recovering the variance function in Section \ref{subsec:MainSteps-SchroEqu2018}, and put the corresponding proofs in Section \ref{subsec:ProofsToMainSteps-SchroEqu2018}.
	
	\subsection{Main unique recovery results} \label{subsec:MainSteps-SchroEqu2018}
	To make it clearer, we use three lemmas, i.e., Lemmas \ref{lem:sigmaHatRec-SchroEqu2018}, \ref{lem:sigmaHatRecErgo-SchroEqu2018} and \ref{lem:sigmaHatRecSingle-SchroEqu2018}, to illustrate our recovering scheme of the variance function. The first main result is as follows.
	\begin{lem} \label{lem:sigmaHatRec-SchroEqu2018}
		We have the following asymptotic identity,
		\begin{equation} \label{eq:sigmaHatRec-SchroEqu2018}
		4\sqrt{2\pi} \lim_{k \to +\infty} \mathbb{E} \Big( \big[ \overline{u^\infty(\hat x, k, \omega)} - \overline{\mathbb{E} u^\infty(\hat x,k)}\, \big] \cdot \big[  u^\infty(\hat x, k+\tau, \omega) - \mathbb{E} u^\infty(\hat x,k+\tau) \big] \Big) = \widehat{\sigma^2}(\tau \hat x),
		\end{equation}
		where $\tau \geq 0,~ \hat x \in \mathbb{S}^2$.
	\end{lem}
	
	Lemma \ref{lem:sigmaHatRec-SchroEqu2018} clearly yields a recovery formula for the variance function. However, it requires many realizations. The result in Lemma \ref{lem:sigmaHatRec-SchroEqu2018} can be improved by using the ergodicity. See, e.g., \cite{caro2016inverse, Lassas2008,  Helin2018}.
	
	\begin{lem} \label{lem:sigmaHatRecErgo-SchroEqu2018}
		Assume $\{K_j\} \in P(2+\gamma)$. Then  $\exists\, \Omega_0 \subset \Omega \colon \mathbb{P}(\Omega_0) = 0$, $\Omega_0$ depends only on $\{K_j\}_{j \in \mathbb{N}^+}$, such that for any $\omega \in \Omega \backslash \Omega_0$, there exists $S_\omega \subset \R^3 \colon m(S_\omega) = 0$, such that for $\forall x \in \R^3 \backslash S_\omega$, 
		\begin{align} 
		& 4\sqrt{2\pi} \lim_{j \to +\infty} \frac 1 {K_j} \int_{K_j}^{2K_j} \big[ \overline{u^\infty(\hat x,k,\omega)} - \overline{\mathbb{E} u^\infty(\hat x,k)}\, \big] \cdot \big[ u^\infty(\hat x,k+\tau,\omega) - \mathbb{E} u^\infty(\hat x,k+\tau) \big] \dif{k} \nonumber\\
		& = \widehat{\sigma^2} (x), \label{eq:SecondOrderErgo-SchroEqu2018}
		\end{align}
		where $\tau = |x|$ and $\hat x := x / |x|$.
	\end{lem}
	
	\sq{The recovering formula \eqref{eq:SecondOrderErgo-SchroEqu2018} holds for any $\hat x \in \mathbb S^2$ when $x = 0$.}
	The recovery formulae presented in Lemma \ref{lem:sigmaHatRecErgo-SchroEqu2018} still involves every realization of the random sample $\omega$. 
	To recover the variance function by only one realization, the term $\mathbb{E} u^\infty(\hat x,k)$ should be further relaxed in Lemma \ref{lem:sigmaHatRecErgo-SchroEqu2018}, and this is achieved by Lemma \ref{lem:sigmaHatRecSingle-SchroEqu2018}.
	
	\begin{lem} \label{lem:sigmaHatRecSingle-SchroEqu2018}
		Under the same condition as in Lemma \ref{lem:sigmaHatRecErgo-SchroEqu2018}, we have
		\begin{equation} \label{eq:sigmaHatRecSingle-SchroEqu2018}
		4\sqrt{2\pi} \lim_{j \to +\infty} \frac 1 {K_j} \int_{K_j}^{2K_j} \overline{u^\infty(\hat x,k,\omega)} \cdot u^\infty(\hat x,k+\tau,\omega) \dif{k} = \widehat{\sigma^2} (x), \quad \ass.
		\end{equation}
	\end{lem}
	
	\sq{
	\begin{rem}
		In Lemma \ref{lem:sigmaHatRecSingle-SchroEqu2018}, it should be noted that the left-hand-side of \eqref{eq:sigmaHatRecSingle-SchroEqu2018} contains the random sample $\omega$, while the right-hand-side does not. This means that the limit in \eqref{eq:sigmaHatRecSingle-SchroEqu2018} is statistically stable.
	\end{rem}}

	Now Theorem \ref{thm:Unisigma-SchroEqu2018} becomes a direct consequence of Lemma \ref{lem:sigmaHatRecSingle-SchroEqu2018}. 
	
	\begin{proof}[Proof of Theorem \ref{thm:Unisigma-SchroEqu2018}]
		Lemma \ref{lem:sigmaHatRecSingle-SchroEqu2018} provides a recovery formula for the variance function $\sigma^2$ by the data set $\mathcal M_1$.
	\end{proof}

	\subsection{Proofs of the main results} \label{subsec:ProofsToMainSteps-SchroEqu2018}
	
	In this subsection, we present proofs of Lemmas \ref{lem:sigmaHatRec-SchroEqu2018}, \ref{lem:sigmaHatRecErgo-SchroEqu2018} and \ref{lem:sigmaHatRecSingle-SchroEqu2018}.

	\begin{proof}[Proof of Lemma \ref{lem:sigmaHatRec-SchroEqu2018}]
		\sq{Write $u_1^\infty(\hat x,k,\omega) = u^\infty(\hat x,k,\omega) - \mathbb{E} u^\infty(\hat x,k)$ as in \eqref{eq:u1-SchroEqu2018}. Therefore $4\pi u_1^\infty(\hat x,k,\omega)$ equals to $(-1) \sum_{j=0}^{+\infty} \int_D e^{-ik \hat x \cdot y} (V \Rk)^j (\sigma \dot{B}_y)\dif{y}$. Recall the definition of $F_j(k,\hat x,\omega)$ $(j = 0,1)$ in \eqref{eq:Fjkx-SchroEqu2018}.}
		Let $k_1, k_2 > k > k^*$. One can compute 
		\begin{align}
		16\pi^2  \mathbb{E} \big( \overline{u_1^\infty(\hat x,k_1,\omega)} u_1^\infty(\hat x,k_2,\omega) \big)
		& = \sq{\sum_{j,\ell = 0,1}} \mathbb{E} \big( \overline{F_j(k_1,\hat x,\omega)} F_\ell(k_2,\hat x,\omega) \big) \nonumber\\
		& =: I_0 + I_1 + I_2 + I_3. \label{eq:Thm1uI-SchroEqu2018}
		\end{align}
		From Lemma \ref{lemma:HOT-SchroEqu2018}, we have
		$I_1,\, I_2,\, I_3$ are all of order $k^{-1}$, hence
		\begin{equation} \label{eq:u1Infty-SchroEqu2018}
		16\pi^2  \mathbb{E} \big( \overline{u_1^\infty(\hat x,k_1,\omega)} u_1^\infty(\hat x,k_2,\omega) \big) = I_0 + \mathcal{O}(k^{-1}), \quad k \to +\infty.
		\end{equation}
		By \eqref{eq:I0-SchroEqu2018}, \eqref{eq:Thm1uI-SchroEqu2018} and \eqref{eq:u1Infty-SchroEqu2018}, we have
		$$16\pi^2 \lim_{k \to +\infty} \mathbb{E} \big( \overline{u_1^\infty (\hat x,k_1,\omega)} u_1^\infty (\hat x,k_2,\omega) \big) = (2\pi)^{3/2} \,\widehat{\sigma^2}((k_2 - k_1) \hat x),$$
		which implies (\ref{eq:sigmaHatRec-SchroEqu2018}).
	\end{proof}
	
	\medskip
	
	\begin{proof}[Proof of Lemma \ref{lem:sigmaHatRecErgo-SchroEqu2018}]
		Our proof is divided into two steps. In the first step we give a basic result, i.e., the conclusion  \eqref{eq:SecondOrderErgo2-SchroEqu2018}, and in the second step the logical order between $y$ and $\omega$ in \eqref{eq:SecondOrderErgo2-SchroEqu2018} is exchanged.
		
		\smallskip
		
		\noindent \textbf{Step 1}: give a basic result.
		
		We denote by $\mathcal{E}_k$ the averaging operation w.r.t. $k$: ${{\mathcal{E}_k f} = \frac 1 K \int_K^{2K} f(k) \dif{k}}$. Following the notation conventions in the proof of Lemma \ref{lem:sigmaHatRec-SchroEqu2018}, we have
		\begin{align}
		16\pi^2 \mathcal{E}_k \big( \overline{u_1^\infty(\hat x,k,\omega)} u_1^\infty(\hat x,k+\tau,\omega) \big)
		& = \jz{\sum_{j,\ell = 0,1}} \mathcal{E}_k \big( \overline{F_j(k,\hat x,\omega)} F_\ell(k+\tau,\hat x,\omega) \big) \nonumber\\
		& =: X_{0,0} + X_{0,1}+ X_{1,0} + X_{1,1}. \label{eq:Thm2uX-SchroEqu2018}
		\end{align}
		Recall that $\{K_j\} \in P(2+\gamma)$.~Then, for $\forall \tau \geq 0$ and $\forall \hat x \in \mathbb{S}^2$, Lemma \ref{lemma:LeadingTermErgo-SchroEqu2018} implies that $\exists\, \Omega_{\tau,\hat x}^{0,0} \subset \Omega \colon \mathbb{P}(\Omega_{\tau,\hat x}^{0,0}) = 0$, $\Omega_{\tau,\hat x}^{0,0}$ depending on $\tau$ and $\hat x$, such that
		\begin{equation} \label{eq:Thm2X00-SchroEqu2018}
		\lim_{j \to +\infty} X_{0,0}(K_j,\tau,\hat x,\omega) = (2\pi)^{3/2} \widehat{\sigma^2} (\tau \hat x), \quad \forall \omega \in \Omega \backslash \Omega_{\tau,\hat x}^{0,0}.
		\end{equation}
		$\{K_j\} \in P(2+\gamma)$ implies $\{K_j\} \in P(5/4+\gamma)$, so Lemma \ref{lemma:HOTErgo-SchroEqu2018} implies the existence of the sets $\Omega_{\tau,\hat x}^{p,q} ~\big( (p,q) \in \{ (0,1),\, (1,0),\, (1,1) \} \big)$ with zero probability measures such that $\forall \tau \geq 0$ and $\forall \hat x \in \mathbb{S}^2$,
		\begin{equation} \label{eq:Thm2Xpq-SchroEqu2018}
		\lim_{j \to +\infty} X_{p,q}(K_j,\tau,\hat x,\omega) = 0, \quad \forall \omega \in \Omega \backslash \Omega_{\tau,\hat x}^{p,q}.
		\end{equation}
		for all $(p,q) \in \{ (0,1),\, (1,0),\, (1,1) \}$. Write $\Omega_{\tau,\hat x} = \bigcup_{p,q = 0,1} \Omega_{\tau,\hat x}^{p,q}$\,, then $\mathbb{P} (\Omega_{\tau,\hat x}) = 0$. From Lemmas \ref{lemma:LeadingTermErgo-SchroEqu2018} and \ref{lemma:HOTErgo-SchroEqu2018} we note that $\Omega_{\tau,\hat x}^{p,q}$ also depends on $K_j$, so does $\Omega_{\tau,\hat x}$, but we omit this dependence in the notation. Write
		\[
		Z(\tau\hat x,\omega) := \lim_{j \to +\infty}  \frac {16\pi^2} {K_j} \int_{K_j}^{2K_j} \overline{u_1^\infty(\hat x,k,\omega)} u_1^\infty(\hat x,k+\tau,\omega) \dif{k} - (2\pi)^{3/2} \widehat{\sigma^2} (\tau \hat x)
		\]
		for short. By (\ref{eq:Thm2uX-SchroEqu2018})-(\ref{eq:Thm2Xpq-SchroEqu2018}), we conclude that,
		\begin{equation} \label{eq:SecondOrderErgo2-SchroEqu2018}
		\Forall y \in \R^3, \Exists \Omega_y \subset \Omega \colon \mathbb P (\Omega_y) = 0, \st \forall\, \omega \in \Omega \backslash \Omega_y,\, Z(y,\omega) = 0.
		\end{equation}
		
		\smallskip
		
		\noindent \textbf{Step 2}: exchange the logical order.
		
		To conclude \eqref{eq:SecondOrderErgo-SchroEqu2018} from \eqref{eq:SecondOrderErgo2-SchroEqu2018}, we should exchange the logical order between $y$ and $\omega$. To achieve this, we utilize Fubini's Theorem. Denote the usual Lebesgue measure on $\R^3$ as $\mathbb L$ and the product measure $\mathbb L \times \mathbb P$ as $\mu$, and construct the product measure space $\mathbb M := (\R^3 \times \Omega, \mathcal G, \mu)$ in the canonical way, where $\mathcal G$ is the corresponding complete $\sigma$-algebra. Write
		\[
		\mathcal{A} := \{ (y,\omega) \in \R^3 \times \Omega \,;\, Z(y, \omega) \neq 0 \},
		\]
		then $\mathcal{A}$ is a subset of $\mathbb M$. Set $\chi_\mathcal{A}$ as the characteristic function of $\mathcal{A}$ in $\mathbb M$. By \eqref{eq:SecondOrderErgo2-SchroEqu2018} we obtain
		\begin{equation} \label{eq:FubiniEq0-SchroEqu2018}
		\int_{R^3} \big( \int_\Omega \chi_{\mathcal A}(y,\omega) \dif{\mathbb P(\omega)} \big) \dif{\mathbb L(y)} = 0.
		\end{equation}
		By \eqref{eq:FubiniEq0-SchroEqu2018} and [Corollary 7 in Section 20.1, \citen{royden2000real}] we obtain
		\begin{equation} \label{eq:FubiniEq1-SchroEqu2018}
		\int_{\mathbb M} \chi_{\mathcal A}(y,\omega) \dif{\mathbb \mu} = \int_\Omega \big( \int_{R^3} \chi_{\mathcal A}(y,\omega) \dif{\mathbb L(y)} \big) \dif{\mathbb P(\omega)} = 0.
		\end{equation}
		Because $\chi_{\mathcal A}(y,\omega)$ is non-negative, (\ref{eq:FubiniEq1-SchroEqu2018}) implies
		\begin{equation} \label{eq:FubiniEq2-SchroEqu2018}
		\Exists \Omega_0 \colon \mathbb P (\Omega_0) = 0, \st \forall\, \omega \in \Omega \backslash \Omega_0,\, \int_{R^3} \chi_{\mathcal A}(y,\omega) \dif{\mathbb L(y)} = 0.
		\end{equation}
		Formula (\ref{eq:FubiniEq2-SchroEqu2018}) further implies for every $\omega \in \Omega \backslash \Omega_0$,
		\begin{equation} \label{eq:FubiniEq3-SchroEqu2018}
		\Exists S_\omega \subset \R^3 \colon \mathbb L (S_\omega) = 0, \st \forall\, y \in \R^3 \backslash S_\omega,\, Z(y,\omega) = 0.
		\end{equation}
		From (\ref{eq:FubiniEq3-SchroEqu2018}) we arrive at (\ref{eq:SecondOrderErgo-SchroEqu2018}).
	\end{proof}
	
	\medskip
	
	\begin{proof}[Proof of Lemma \ref{lem:sigmaHatRecSingle-SchroEqu2018}]
		The symbol $\mathcal{E}_k$ is defined the same as in the proof of Lemma \ref{lem:sigmaHatRecErgo-SchroEqu2018}. We have
		\begin{align}
		& 16\pi^2 \mathcal{E}_k \big( \overline{u^\infty(\hat x,k,\omega)} u^\infty(\hat x,k+\tau,\omega) \big) \nonumber\\
		= \ & 16\pi^2 \mathcal{E}_k \big( \overline{u_1^\infty(\hat x,k,\omega)} \cdot u_1^\infty(\hat x,k,\omega) \big) + 16\pi^2 \mathcal{E}_k \big( \overline{u_1^\infty(\hat x,k,\omega)} \cdot \mathbb E u^\infty(\hat x,k+\tau) \big) \nonumber\\
		& + 16\pi^2 \mathcal{E}_k \big( \overline{\mathbb E u^\infty(\hat x,k)} \cdot u_1^\infty(\hat x,k+\tau,\omega) \big) + 16\pi^2 \mathcal{E}_k \big( \overline{\mathbb E u^\infty(\hat x,k)} \cdot \mathbb E u^\infty(\hat x,k+\tau) \big) \nonumber\\
		=: & J_0 + J_1 + J_2 + J_3. \label{eq:J-SchroEqu2018}
		\end{align}
		
		From Lemma \ref{lem:sigmaHatRecErgo-SchroEqu2018} we obtain
		\begin{equation} \label{eq:J0-SchroEqu2018}
		\begin{split}
		& \lim_{j \to +\infty} J_0  = \lim_{j \to +\infty} \jz{\frac {16\pi^2} {K_j}} \int_{K_j}^{2K_j} \overline{u_1^\infty(\hat x,k,\omega)} \cdot u_1^\infty(\hat x,k+\tau,\omega) \dif{k} = (2\pi)^{3/2} \widehat{\sigma^2} (\tau \hat x), \\
		& \quad \tau \hat x \aee \! \in \R^3, \quad \omega \ass \! \in \Omega.
		\end{split}
		\end{equation}
		
		We now estimate $J_1$,
		\begin{align}
		|J_1|^2
		& \simeq \big| \mathcal{E}_k \big( \overline{u_1^\infty(\hat x,k,\omega)} \cdot \mathbb E u^\infty(\hat x,k+\tau) \big) \big|^2
		= \big| \frac 1 {K_j} \int_{K_j}^{2K_j} \overline{u_1^\infty(\hat x,k,\omega)} \cdot \mathbb E u^\infty(\hat x,k+\tau) \dif{k} \big|^2 \nonumber\\
		& \leq \frac 1 {K_j} \int_{K_j}^{2K_j} |u^\infty(\hat x,k,\omega) - \mathbb{E} u^\infty(\hat x,k)|^2 \dif{k} \cdot \frac 1 {K_j} \int_{K_j}^{2K_j} |\mathbb E u^\infty(\hat x,k+\tau)|^2 \dif{k}. \label{eq:J1One-SchroEqu2018}
		\end{align}
		Combining \eqref{eq:J1One-SchroEqu2018} with Lemmas \ref{lemma:FarFieldGoToZero-SchroEqu2018} and \ref{lem:sigmaHatRecErgo-SchroEqu2018}, we have
		\begin{equation} \label{eq:J1-SchroEqu2018}
		|J_1|^2 \lesssim (\widehat{\sigma^2}(0) + o(1)) \cdot o(1) = o(1) \to 0, \quad j \to +\infty.
		\end{equation}
		The analysis of $J_2$ is similar to that of $J_1$ so we skip the details.

		\smallskip

		Finally, by Lemma \ref{lemma:FarFieldGoToZero-SchroEqu2018}, the $J_3$ can be estimated as
		\begin{align}
		|J_3|^2
		& \simeq \big| \mathcal{E}_k \big( \overline{\mathbb E u^\infty(\hat x,k)} \cdot \mathbb E u^\infty(\hat x,k+\tau) \big) \big|^2 \nonumber\\
		& \leq \frac 1 {K_j} \int_{K_j}^{2K_j} \sup_{\kappa \geq K_j} \big| \mathbb E u^\infty(\hat x,\kappa) \big|^2 \dif{k} \cdot \frac 1 {K_j} \int_{K_j}^{2K_j} \sup_{\kappa \geq K_j+\tau} \big| \mathbb E u^\infty(\hat x,\kappa) \big|^2 \dif{k} \nonumber\\
		& = \sup_{\kappa \geq K_j} |\mathbb E u^\infty(\hat x,\kappa)|^2 \cdot \sup_{\kappa \geq K_j+\tau} |\mathbb E u^\infty(\hat x,\kappa)|^2 \to 0, \quad j \to +\infty. \label{eq:J3-SchroEqu2018}
		\end{align}
		Combining \eqref{eq:J-SchroEqu2018}, \eqref{eq:J0-SchroEqu2018}, \eqref{eq:J1-SchroEqu2018} and \eqref{eq:J3-SchroEqu2018}, we arrive at \eqref{eq:sigmaHatRecSingle-SchroEqu2018}. Our proof is done.
	\end{proof}

	\section{Uniqueness of the potential and the random source} \label{sec:RecPS-SchroEqu2018}
	
	In this section, we focus on the recovery of the potential term and the expectation of the random source. Due to the highly nonlinear relation between the total wave and the potential, the active scattering measurements are thus utilized to recover the potential. 
	In the recovery of the potential, the random sample $\omega$ is set to be fixed so that a single realization of the random term $\dot B_x$ is enough to obtain the unique recovery. 
	Different from the recovery of the potential, the uniqueness of the expectation requires all realizations of the random sample $\omega$. 
	\jz{Because the deterministic and random parts of the source are entangled together so that only one realization of the random source cannot reveal exact values of the expectation at each spatial point $x$.} 
	
	\subsection{Recovery of the potential}
	
	Now we are in the position to prove Theorem \ref{thm:UniPot1-SchroEqu2018}. We are to use the incident plane wave, so $\alpha$ is set to be 1 throughout this section.
	
	\begin{proof}[Proof of Theorem \ref{thm:UniPot1-SchroEqu2018}]
		The random sample $\omega$ is assumed to be fixed. Given two direction $d_1$ and $d_2$ of the incident plane waves, we denote the corresponding total wave as $u_{d_1}$ and $u_{d_2}$, respectively. Then, from \eqref{eq:1}, we have
		\begin{equation} \label{eq:uSubtract-SchroEqu2018}
		\begin{cases}
		(-\Delta - k^2)(u_{d_1} - u_{d_2}) = V (u_{d_1} - u_{d_2}) & \\
		u_{d_1} - u_{d_2} = e^{ikd_1 \cdot x} - e^{ikd_2 \cdot x} + u_{d_1}^{sc}(x) - u_{d_2}^{sc}(x) & \\
		u_{d_1}^{sc}(x) - u_{d_2}^{sc}(x): \text{ SRC} &
		\end{cases}
		\end{equation}
		From \eqref{eq:uSubtract-SchroEqu2018} we have the Lippmann-Schwinger equation,
		\begin{equation} \label{eq:uLippSchw-SchroEqu2018}
		\big( I - \Rk V \big) (u_{d_1} - u_{d_2}) = e^{ikd_1 \cdot x} - e^{ikd_2 \cdot x}.
		\end{equation}
		When $k > k^*$, equality \eqref{eq:uLippSchw-SchroEqu2018} gives
		\[
		u_{d_1}^{sc} - u_{d_2}^{sc} 
		= \Rk V (e^{ikd_1 \cdot x} - e^{ikd_2 \cdot x}) + \sum_{j=2}^\infty (\Rk V)^j (e^{ikd_1 \cdot x} - e^{ikd_2 \cdot x}).
		\]
		Therefore the difference between the far-field patterns is
		\begin{align}
		& u^{\infty}(\hat{x},k,d_1) - u^{\infty}(\hat{x},k,d_2) \nonumber\\
		= \ & \int_{D} \frac{e^{-ik\hat{x} \cdot y}}{4\pi} V(y) (e^{ikd_1 \cdot y} - e^{ikd_2 \cdot y}) \dif{y} + \sum_{j=1}^\infty \int_{D} \frac{e^{-ik\hat{x} \cdot y}}{4\pi} V(y) (\Rk V)^j (e^{ikd_1 \cdot (\cdot)} - e^{ikd_2 \cdot (\cdot)}) \dif{y} \nonumber\\
		=: & \sqrt{\frac{\pi}{2}} \widehat{V} \big( k(\hat{x} - d_1) \big) - \sqrt{\frac{\pi}{2}} \widehat{V} \big( k(\hat{x} - d_2) \big) + \sum_{j=1}^\infty H_j(k), \label{eq:uFarfied-SchroEqu2018}
		\end{align}
		where
		\begin{equation} \label{eq:Fjk-SchroEqu2018}
		H_j(k) := \int_{D} \frac{e^{-ik\hat{x} \cdot y}}{4\pi} V(y) (\Rk V)^j (e^{ikd_1 \cdot (\cdot)} - e^{ikd_2 \cdot (\cdot)}) \dif{y}, \quad j = 1,2, \cdots.
		\end{equation}
		For any $p \in \R^3$, when $p = 0$, we let $\hat x = (1,0,0),$ $d_1 = (1,0,0),$ $d_2 = (0,1,0)$; when $p \neq 0$, we can always find a $p^\perp \in \R^3$ which is perpendicular to $p$. Let
		\begin{equation*} 
		e = p^\perp / \nrm{p^\perp}
		\quad\text{ and }\quad
		\left\{\begin{aligned}
		\hat x & = \sqrt{1 - \nrm{p}^2 / (4k^2)} \cdot e + p / (2k), \\
		d_1    & = \sqrt{1 - \nrm{p}^2 / (4k^2)} \cdot e - p / (2k), \\
		d_2    & = p/\nrm{p},
		\end{aligned}\right.
		\end{equation*}
		when $k > \nrm{p}/2$, we have
		\begin{equation} \label{eq:xhatd1Property-SchroEqu2018}
		\left\{\begin{aligned}
		& \hat x, d_1, d_2 \in \mathbb{S}^2, \\
		& k(\hat{x} - d_1) = p, \\
		& |k(\hat{x} - d_2)| \to \infty ~(k \to \infty).
		\end{aligned}\right.
		\end{equation}
		Note that the choices of these two unit vectors $\hat x$, $d_1$ depend on $k$. For different values of $k$, we pick up different directions $\hat x$, $d_1$ to guarantee \eqref{eq:xhatd1Property-SchroEqu2018}. Then,
		\begin{equation} \label{eq:Vxd1d2-SchroEqu2018}
		\sqrt{\frac \pi 2}\widehat{V}(p) = \lim_{k \to +\infty} \big( \sqrt{\frac \pi 2} \widehat{V} (k(\hat{x} - d_1)) - \sqrt{\frac{\pi}{2}} \widehat{V} (k(\hat{x} - d_2)) \big).
		\end{equation}
		Combining \eqref{eq:uFarfied-SchroEqu2018}, \eqref{eq:Vxd1d2-SchroEqu2018} and Lemma \ref{lemma:FjkEstimated-SchroEqu2018}, we conclude
		\begin{equation} \label{eq:PotnFourier-SchroEqu2018}
		\widehat{V}(p) = \sqrt{\frac{2}{\pi}} \lim_{k \to +\infty} \big( u^{\infty}(\hat{x},k,d_1) - u^{\infty}(\hat{x},k,d_2) \big).
		\end{equation}
		Formula \eqref{eq:PotnFourier-SchroEqu2018} completes the proof.
	\end{proof}
	
	It remains to give the estimates of these high-order terms $H_j(k)$, and this is done by Lemma \ref{lemma:FjkEstimated-SchroEqu2018}.
	\begin{lem} \label{lemma:FjkEstimated-SchroEqu2018}
		The sum of high-order terms $H_j(k)$ defined in \eqref{eq:Fjk-SchroEqu2018} satisfies the following estimate,
		$$\big| \sum_{j \geq 1} H_j(k) \big| \leq C k^{-1},$$
		for some constant $C$ independent of $k$.
	\end{lem}
	\begin{proof}[Proof of Lemma \ref{lemma:FjkEstimated-SchroEqu2018}]
		According to Lemma \ref{lemma:RkVBounded-SchroEqu2018}, we have
		\begin{align*}
		|H_j(k)|
		& \lesssim \int_D |V(y)| \cdot \big| [(\Rk V)^{j} e^{ik d_1 \cdot (\cdot)}] (y) \big| \dif{y} + \int_D |V(y)| \cdot \big| [(\Rk V)^{j} e^{ik d_2 \cdot (\cdot)}] (y) \big| \dif{y} \\
		& \lesssim \nrm[L^\infty]{V} \cdot |D|^{1/2} \cdot \big( k^{-j} \nrm[L^2(D)]{ e^{ik d_1 \cdot (\cdot)} } + k^{-j} \nrm[L^2(D)]{ e^{ik d_2 \cdot (\cdot)} } \big) \\
		& = 2\nrm[L^\infty]{V} \cdot |D| \cdot k^{-j}.
		\end{align*}
		Therefore,
		$$|\sum_{j=1}^{\infty} H_j(k)|
		\leq \sum_{j=1}^{\infty} |H_j(k)|
		\leq 2C \nrm[L^\infty]{V} \cdot |D| \cdot \sum_{j=1}^{\infty} k^{-j} \leq C k^{-1}, \quad k \to +\infty.$$
		The proof is done.
	\end{proof}

	\subsection{Recovery of the random source}
	
	The variance function of the random source is recovered in Section \ref{sec:RecVar-SchroEqu2018}, and now we recover its expectation.
	
	\begin{proof}[Proof to Theorem \ref{thm:UniSou1-SchroEqu2018}]
		According to Theorem \ref{thm:UniPot1-SchroEqu2018}, we have the uniqueness of the potential. Assume that two source $f$, $f'$ generate same far-field patterns for all $k > 0$. We denote the restriction on $D$ of the corresponding total waves as $u$ and $u'$. Then,
		\begin{equation} \label{eq:uuprime-SchroEqu2018}
		\left\{\begin{aligned}
		(\Delta + k^2 + V) (\mathbb{E} u - \mathbb{E} u') & = f - f' && \text{ in } D \\
		\mathbb{E} u - \mathbb{E} u' = \partial_\nu (\mathbb{E} u) - \partial_\nu (\mathbb{E} u') & = 0 && \text{ on } \partial D
		\end{aligned}\right.
		\end{equation}
		where $\nu$ is the outer normal to $\partial D$. Let test functions $v_k \in H_0^1(D)$ be the weak solutions of the boundary value problem
		\begin{equation} \label{eq:DiriLap-SchroEqu2018}
		\left\{\begin{aligned}
		(-\Delta - V)v_k & = k^2 v_k && \text{ in } D \\
		v_k & = 0 && \text{ on } \partial D
		\end{aligned}\right.
		\end{equation}
		for delicately picked $k$. The solutions $v_k$ are eigenvectors of the system \eqref{eq:DiriLap-SchroEqu2018}. From \eqref{eq:uuprime-SchroEqu2018} we have
		\begin{equation} \label{eq:uuprimev-SchroEqu2018}
		\int_D (\Delta + V  + k^2) (\mathbb{E}u - \mathbb{E}u') \cdot v_k \dif{x} = \int_D (f - f')v_k \dif{x}.
		\end{equation}
		Using integral by parts and noting that the $v_k$'s in \eqref{eq:uuprimev-SchroEqu2018} satisfy \eqref{eq:DiriLap-SchroEqu2018}, we have
		\begin{equation} \label{eq:ffprime-SchroEqu2018}
		\int_D (f - f')v_k \dif{x} = 0.
		\end{equation}
		When $\nrm[L^\infty(D)]{V}$ is less than some constant depending on $D$, the set of eigenvectors $\{v_k\}$ corresponding to different eigenvalues $k^2$ forms an orthonormal basis of $L^2(D)$ [Theorem 2.37, \citen{mclean2000strongly}]. Therefore, from \eqref{eq:ffprime-SchroEqu2018} we conclude that
		$$f = f' \text{ in } L^2(D).$$
		The proof is done.
	\end{proof}

	\section{Conclusions} \label{sec:Conclusions-SchroEqu2018}
	
	In this paper, we are concerned with a random Schr\"odinger equation. 
	First, the well-posedness of the direct problem is studied. 
	Then, the variance function of the random source is recovered by using a single passive scattering measurement. 
	By further utilizing active scattering measurements under a single realization of the random sample, the potential is recovered. 
	Finally, with the help of multiple realizations of the random sample, the expectation of the random source are recovered. 
	The major novelty of our study is that on the one hand, both the random source and the potential are unknown, and on the other hand, both passive and active measurements are used to recover all of the unknowns. 
	
	\sq{While the direct problem in this paper is well-formulated in the space $L_{-1/2-\epsilon}^2$, the regularity of the solution of the random Schr\"odinger system is not taken into consideration. 
	A different formulation of the direct problem, which takes the regularity of the solution into consideration, is possible. 
	And this new formulation gives possibility to handle the case where both the source and potential are random. 
	We shall report our finding in this aspect in a forthcoming article.
	}
	
	\sq{
		\section*{Acknowledgements}

		We are grateful to the two anonymous referees and the handling editor for many constructive comments and suggestions, which have led to significant improvements on the results and presentation of the paper. The work of J.~Li was partially supported by the NSF of China under the grant No.~11571161 and 11731006, the Shenzhen Sci-Tech Fund No.~JCYJ20170818153840322. The work of H.~Liu was partially supported by HKBU FRG Funds and Hong Kong RGC Grants, No. 12302017 and 12301218}.


\end{document}